\documentclass[12pt]{amsart}
\usepackage{amsfonts}
\usepackage{amsmath}
\usepackage{amssymb}
\usepackage[utf8]{inputenc}
\usepackage[english]{babel}
\usepackage{enumerate}
\usepackage[a4paper, total={6.6in, 9in}]{geometry}
\usepackage{mathrsfs}
\usepackage{graphicx}

\theoremstyle{plain}

\newtheorem{theorem}{Theorem}[section]
\newtheorem{corollary}[theorem]{Corollary}
\newtheorem{lemma}[theorem]{Lemma}
\newtheorem{proposition}[theorem]{Proposition}

\newtheorem{remark}[theorem]{Remark}
\theoremstyle{definition}

\newtheorem{acknowledgement}[theorem]{Acknowledgement}

\newtheorem{example}[theorem]{Example}

\numberwithin{equation}{section}

\DeclareMathOperator{\R}{\mathbb R}

\input{tcilatex}

\begin{document}
\title[On minimizing surfaces of the CR invariant energy $E_1$ ]{ On
minimizing surfaces of the CR invariant energy $E_1$}
\author{Jih-Hsin Cheng}
\address{Institute of Mathematics, Academia Sinica and National Center for
Theoretical Sciences, Taipei, Taiwan, R.O.C.}
\email{cheng@math.sinica.edu.tw}
\author{Hung-Lin Chiu}
\address{Department of Mathematics, National Tsing-Hua University and
National Center for Theoretical Sciences, Hsinchu, Taiwan, R.O.C.}
\email{hlchiu@math.nthu.edu.tw}
\author{Paul Yang}
\address{Department of Mathematics, Princeton University, Princeton, NJ
08544, U.S.A.}
\email{yang@Math.Princeton.EDU}
\author{Yongbing Zhang}
\address{School of Mathematical Sciences and Wu Wen-Tsun Key Laboratory of
Mathematics, University of Science and Technology of China, Hefei, Anhui,
230026, P. R. of China.}
\email{ybzhang@amss.ac.cn}
\thanks{}

\begin{abstract}
We study a CR-invariant equation for vanishing $E_{1}$ surfaces in the
3-dimensional Heisenberg group. This is shown to be a hyperbolic equation.
We prove the local uniqueness theorem for an initial value problem and
classify all such global surfaces with rotational symmetry. We also show
that the Clifford torus in the CR 3-sphere is not a local minimizer of $E_1$
by computing the second variation.
\end{abstract}

\maketitle



\section{Introduction and statement of the results}

Let $(M,\xi )$ be a contact $3$-manifold $M$ with contact structure $\xi $.
A CR structure is an endomorphism $J:\xi \rightarrow \xi $ such that $%
J^{2}=-I$. Consider a (smooth: meaning \textquotedblleft $C^{\infty }$%
-smooth" unless specified otherwise) surface $\Sigma $ in the CR manifold ($%
M,\xi ,J)$. A CR invariant surface element is a $2$-form $dA_{\Sigma }$
defined on $\Sigma $ such that 
\begin{equation*}
dA_{\Sigma _{1}}=\varphi ^{\ast }(dA_{\Sigma _{2}})
\end{equation*}%
for any (local) CR diffeomorphism $\varphi $ from a neighborhood $U_{1}$ of $%
\Sigma _{1}$ to a neighborhood $U_{2}$ of $\Sigma _{2}$ such that $\varphi
(\Sigma _{1})=\Sigma _{2}$. In the paper \cite{C} of early 90's, one of the
authors discovered two CR invariant surface elements $dA_{1}$ and $dA_{2}$,
via Cartan-Chern's method of admissible frames. In a subsequent paper \cite%
{CYZ18}, three of the authors show that for a nonsingular ($T\Sigma \neq \xi 
$) surface $\Sigma \subset M$, they can express $dA_{1}$ and $dA_{2}$ by
geometric quantities $e_{1}$ (characteristic unit tangent), $\alpha $ (the
deviation function)$,$ $H$ (the $p$-mean curvature) obtained from the
ambient pseudohermitian structure $(J,\Theta )$ (see a short review for
related quantities in Section \ref{Sec2}). In particular, we have (see \cite[%
Theorem 2]{CYZ18} or Proposition \ref{A-P-2} in Section \ref{Sec2}) 
\begin{equation}
dA_{1}=|e_{1}\alpha +\frac{1}{2}\alpha ^{2}+\frac{1}{6}H^{2}+\frac{1}{4}W-%
\func{Im}A_{11}|^{3/2}\Theta \wedge e^{1}.  \label{1-0}
\end{equation}%
\noindent Therefore we can define the CR invariant energy functionals:
(assuming the set of singular points has measure 0 in $\Sigma $) 
\begin{equation*}
E_{1}(\Sigma )=\int_{\Sigma }dA_{1}.
\end{equation*}%
First note that $E_{1}(\Sigma )=0$ for a (smooth) surface $\Sigma \subset M$%
=the Heisenberg group $H_{1}$ if and only if the following equation holds
true (on the nonsingular domain where $e_{1},$ $\alpha ,$ $H$ are defined;
for $H_{1}$, the Webster (scalar) curvature $W=0$ and the torsion $A_{11}=0$%
): 
\begin{equation}
e_{1}\alpha +\frac{1}{2}\alpha ^{2}+\frac{1}{6}H^{2}=0.  \label{enzeeq1}
\end{equation}

\noindent We often call $\Sigma $ or its graph (smooth) function $u$ an
absolute $E_{1}$ energy minimizer or absolute $E_{1}$-minimizer (with
vanishing energy).

On the region where $\alpha \neq 0$, a surface can be defined by the graph
of a (smooth) function $t=u(x,y)$ on a (nonsingular: $D\neq 0$) domain of
the $xy$-plane. In terms of $u(x,y)$, equation (\ref{enzeeq1}) reads as 
\begin{equation}
\begin{split}
0& =F(x,y,u,u_{x},u_{y},u_{xx},u_{xy},u_{yy}) \\
& =\left( \frac{u_{y}+x}{D}\right) D_{x}-\left( \frac{u_{x}-y}{D}\right)
D_{y}-\frac{1}{6}\left( \Delta u-\left( \frac{u_{x}-y}{D}\right)
D_{x}-\left( \frac{u_{y}+x}{D}\right) D_{y}\right) ^{2}-\frac{1}{2}
\end{split}
\label{enzeeq7}
\end{equation}%
\noindent where $D$ $:=$ $[(u_{x}-y)^{2}+(u_{y}+x)^{2}]^{1/2}$ (see Section %
\ref{Sec2} for the deduction). By Proposition \ref{A-P-1} in Section \ref%
{Sec2}, we have%
\begin{equation*}
F_{u_{xx}}F_{u_{yy}}-\frac{1}{4}F_{u_{xy}}^{2}=-\frac{1}{4}<0,
\end{equation*}%
\noindent meaning that (\ref{enzeeq7}) is a second-order hyperbolic equation
in the sense that the linearized operator is hyperbolic (see, for instance, 
\cite[(89) on p.398]{Eva98}). 
Let $(r,\phi )$ denote the polar coordinates of the $xy$-plane. Write 
\begin{equation}
\cos \theta =\frac{u_{x}-y}{D},\quad \sin \theta =\frac{u_{y}+x}{D}.
\label{1-1}
\end{equation}%
\noindent For a curve to be characteristic, see (\ref{chaeq}). We have the
uniqueness of absolute $E_{1}$-minimizers for the following initial-value
problem.

\begin{theorem}
\label{mainthm2} Suppose that on a nonsingular domain $u=u(r,\phi )$ and $%
v=v(r,\phi )$ are two (smooth) absolute $E_{1}$-minimizers having the same
initial values $H,\alpha $ and $\theta $ on a non-characteristic (smooth)
curve $r=\mathbf{c}>0$ with respect to $u\ (or\ U, see\ (\ref{U}))$, lying
on the region where $\det {A(r,\phi ,U)}\neq 0$ (see the notation $\det {%
A(r,\phi ,U)}$ in (\ref{Aex})). Then $u=v$ on a neighborhood of the initial
curve $r=\mathbf{c} $.
\end{theorem}

Note that $\theta ,$ $\alpha $ involve first derivatives of $u$ while $H$
involves $u$-derivatives of second order. The $E_{1}=0$ equation (\ref%
{enzeeq1}) and the integrability condition (\ref{iden2}) form the first two
equations of (\ref{enzeeq2}) via (\ref{emini1}). Notice that $H$ is involved
as a coefficient in this two-equations system, so it is not a quasi-linear
system. Therefore we need to enlarge the set of variables including $H$ ($%
=-N^{\perp }\theta ),$ but then $N^{\perp }\alpha $ should also be included
to make the system \textquotedblleft perfect" (so that the Coddazi-like
equation (\ref{codeq1}) also plays its role). We finally need to consider a
system of four equations, which is quasi-linear; see (\ref{enzeeq2}). This
special structure for a second order nonlinear hyperbolic equation arising
from CR geometry is intriguing and worth further study in the future.

Next we classify all complete (say, with respect to the metric induced from
Euclidean $\mathbb{R}^{3})$, rotationally symmetric $E_{1}$-minimizer with
vanishing energy, i.e. satisfying (\ref{enzeeq1}).

\begin{theorem}
\label{T-class} With the notation above, we conclude that the complete
(smooth), rotationally symmetric absolute $E_{1}$-minimizers are divided
into four classes up to CR transformations on the Heisenberg group $H_{1}$:

\begin{enumerate}
\item $u(r)=\frac{\sqrt{3}}{2}r^{2}$,

\item $u(r)=-\frac{\sqrt{3}}{2}r^{2}$,

\item Type I shifted Heisenberg spheres: $\left( r^{2}+\frac{\sqrt{3}}{2}%
\rho _{0}^{2}\right) ^{2}+4u(r)^{2}=\rho _{0}^{4}$ for $\rho _{0}>0,$

\item Type II shifted Heisenberg spheres: $\left( r^{2}-\frac{\sqrt{3}}{2}%
\rho _{0}^{2}\right) ^{2}+4u(r)^{2}=\rho _{0}^{4}$ for $\rho _{0}>0$.
\end{enumerate}
\end{theorem}

We can obtain the solutions $u(r)=\pm \frac{\sqrt{3}}{2}r^{2}$ in (1)/(2) of
Theorem \ref{T-class} as a limit of a family of dilations applied to type
I/II shifted Heisenberg spheres. See Remark \ref{R-3-1} for a detailed
explanation. One may ask whether type I and type II shifted Heisenberg
spheres are the only closed (compact with no boundary), (smooth) absolute $%
E_{1}$-minimizers up to CR transformations on $H_{1}$ with no assumption on
rotational symmetry.

The Clifford torus in the standard CR 3-sphere is a critical point of $E_1$
and we had conjectured that it is a global minimizer for $E_1$ among all
surfaces of torus type \cite{CYZ18}. However it turns out that the Clifford
torus is actually not a local minimizer for $E_1$. We show this by computing
the second variation of $E_{1}$ for the Clifford torus. Let $F_{t}:\Sigma
\rightarrow S^{3}$ be a family of immersions such that $F_{t=0}=$inclusion
and%
\begin{equation*}
\frac{d}{dt}F_{t}=X=fe_{2}+gT,
\end{equation*}
\noindent where $T$ is the Reeb vector field, $e_{2}=Je_{1}$ for $e_{1}$
being a unit vector in $T\Sigma _{t}\cap \xi $, $\Sigma _{t}:=F_{t}(\Sigma )$%
. We assume $f$ and $g$ are supported in a domain of $\Sigma $ away from the
singular set of $\Sigma $. We have the following second variation formula:

\begin{theorem}
\label{T-2nd-var} Let $\Sigma $ be the Clifford torus $\Sigma _{\lbrack 1/%
\sqrt{2}]}$ in the CR 3-sphere $S^{3}$. Then it holds that%
\begin{equation}
\frac{d^{2}}{dt^{2}}|_{t=0}E_{1}(\Sigma _{t})=\frac{\sqrt{2}}{4}\int_{\Sigma
_{\lbrack 1/\sqrt{2}%
]}}[e_{1}e_{1}(f)^{2}+3e_{1}T(f)^{2}-7e_{1}(f)^{2}+9fTT(f)+12f^{2}]\Theta
\wedge e^{1}.  \label{SVF}
\end{equation}
\end{theorem}

Recall (see Subsection \ref{Sub4-3}) that for the Clifford torus $\Sigma
_{\lbrack 1/\sqrt{2}]}$, $e_{1}=-\frac{\partial }{\partial \varphi _{1}}+%
\frac{\partial }{\partial \varphi _{2}},$ $V=T=\partial _{\varphi
_{1}}+\partial _{\varphi _{2}},$ $z=(z_{1},z_{2})=\frac{1}{\sqrt{2}}%
(e^{i\varphi _{1}},e^{i\varphi _{2}}).$ We may take $f(z)=v_{l}(z)=\cos
l(\varphi _{1}+\varphi _{2}),\quad l=1,2,\cdots .$ Then $%
TT(v_{l})=-4l^{2}v_{l},$ $e_{1}(v_{l})=0$ and for 
\begin{equation*}
\frac{d}{dt}\Sigma _{t}=v_{l}e_{2},
\end{equation*}%
\noindent we have $f=v_{l}$ and 
\begin{eqnarray*}
\frac{d^{2}}{dt^{2}}|_{t=0}E_{1}(\Sigma _{t}) &=&\frac{\sqrt{2}}{4}%
\int_{\Sigma _{\lbrack 1/\sqrt{2}]}}[9v_{l}TT(v_{l})+12v_{l}^{2}]\Theta
\wedge e^{1} \\
&=&3\sqrt{2}\int_{\Sigma _{\lbrack 1/\sqrt{2}]}}(1-3l^{2})v_{l}^{2}\Theta
\wedge e^{1} \\
&<&0.
\end{eqnarray*}

\begin{corollary}
\label{C-1-1} The Clifford torus $\Sigma _{\lbrack 1/\sqrt{2}]}$ is not an $%
E_{1}$-minimizer among all surfaces of torus type in the CR 3-sphere $S^{3}$.
\end{corollary}

We wonder what $E_{1}$-minimizers are if exist in the above corollary. 

The paper is organized as follows. In Section \ref{Sec2}, we show the
hyperbolicity of equation (\ref{enzeeq7}) and the CR invariance of $dA_{1}$
(see (\ref{1-0})). In Section \ref{Sec3}, we prove Theorem \ref{mainthm2}.
In Section \ref{Sec4c}, we show the classification theorem Theorem \ref%
{T-class}. Finally Theorem \ref{T-2nd-var} is proved in Section \ref{Sec4}.

\bigskip

\begin{acknowledgement}
J.-H. Cheng (resp. H.-L. Chiu) would like to thank the National Science and
Technology Council of Taiwan for the support: grant no. 112-2115-M-001-012
(resp. 112-2115-M-007-009-MY3). J.-H. Cheng would also like to thank the
National Center for Theoretical Sciences for the constant support. P. Yang
acknowledges support from the NSF for the grant DMS 1509505. Y. Zhang
acknowledges support from National Key Research and Development Project
SQ2020YFA070080 and NSFC 12071450.
\end{acknowledgement}

\section{Hyperbolicity and CR invariance\label{Sec2}}

We refer basic notations such as contact form $\Theta ,$ characteristic unit
tangent $e_{1},$ $p$-mean curvature $H,$ Legendrian normal $e_{2},$ Reeb
vector field $T$, \textquotedblleft the deviation function" $\alpha $, etc.
to \cite[Section 2]{CHMY05}. For a (nonsingular: $D\neq 0$) graph $t=u(x,y)$
in the Heisenberg group, We can express $e_{1},$ $\alpha ,$ $H$ as follows:%
\begin{eqnarray}
e_{1} &=&-\frac{u_{y}+x}{D}(\frac{\partial }{\partial x}+y\frac{\partial }{%
\partial t})+\frac{u_{x}-y}{D}(\frac{\partial }{\partial y}-x\frac{\partial 
}{\partial t}),  \label{A-1} \\
\alpha  &=&\mp 1/D,\text{ }D:=[(u_{x}-y)^{2}+(u_{y}+x)^{2}]^{1/2},  \notag \\
H
&=&D^{-3}\{(u_{y}+x)^{2}u_{xx}-2(u_{y}+x)(u_{x}-y)u_{xy}+(u_{x}-y)^{2}u_{yy}%
\}  \notag
\end{eqnarray}

\noindent (\cite[p.137 and (2.10)]{CHMY05}). Substituting (\ref{A-1}) into (%
\ref{enzeeq1}) yields (\ref{enzeeq7}$)$. We then observe that (\ref{enzeeq7}$%
)$ is a hyperbolic equation in the sense that the linearized operator is
hyperbolic \cite[(89) on p.398]{Eva98}.

\begin{proposition}
\label{A-P-1} Let $F=F(x,y,u,u_{x},u_{y},u_{xx},u_{xy},u_{yy})$ be as in (%
\ref{enzeeq7}). Then we have%
\begin{equation}
F_{u_{xx}}F_{u_{yy}}-\frac{1}{4}F_{u_{xy}}^{2}=-\frac{1}{4}<0  \label{hyper}
\end{equation}
\end{proposition}

\begin{proof}
From (\ref{1-1}) we recall that $\theta $ satisfies%
\begin{equation}
\cos \theta =\frac{u_{x}-y}{D},\sin \theta =\frac{u_{y}+x}{D}.  \label{A-2}
\end{equation}

\noindent In terms of $\theta ,$ $D$ and $H,$ we write (see (\ref{enzeeq7})
for $F$)%
\begin{equation}
F=(\sin \theta )D_{x}-(\cos \theta )D_{y}-\frac{1}{6}D^{2}H^{2}-\frac{1}{2}
\label{A-3}
\end{equation}

\noindent where we have used 
\begin{eqnarray}
A & {:=} &\Delta u-(\cos \theta )D_{x}-(\sin \theta )D_{y}  \label{A-4} \\
&=&(\sin ^{2}\theta )u_{xx}-(2\sin \theta \cos \theta )u_{xy}+(\cos
^{2}\theta )u_{yy}\overset{(\ref{A-1})}{=}DH.  \notag
\end{eqnarray}

\noindent By the chain rule we easily get%
\begin{eqnarray}
D_{x} &=&(\cos \theta )u_{xx}+(\sin \theta )(u_{yx}+1),  \label{A-5} \\
D_{y} &=&(\cos \theta )(u_{xy}-1)+(\sin \theta )u_{yy}.  \notag
\end{eqnarray}

\noindent From (\ref{A-3}), (\ref{A-5}) and (\ref{A-4}), we can then compute%
\begin{equation}
F_{u_{xx}}=\sin \theta \cos \theta -\frac{1}{6}\cdot 2AA_{u_{xx}}=\sin
\theta \cos \theta -\frac{1}{3}A\sin ^{2}\theta .  \label{A-6a}
\end{equation}

\noindent Similarly, we also obtain%
\begin{eqnarray}
F_{u_{yy}} &=&-\cos \theta \sin \theta -\frac{1}{3}A\cos ^{2}\theta ,
\label{A-6b} \\
F_{u_{xy}} &=&F_{u_{yx}}=\sin ^{2}\theta -\cos ^{2}\theta +\frac{2}{3}A\sin
\theta \cos \theta .  \notag
\end{eqnarray}

\noindent Finally, a direct computation using (\ref{A-6a}) and (\ref{A-6b})
yields (\ref{hyper}).
\end{proof}

We recall the CR invariance of the surface area element $dA_{1}.$

\begin{proposition}
\label{A-P-2} (\cite[Theorem 2]{CYZ18}) On a domain of nonsingular points of
a surface, the surface area element%
\begin{equation*}
dA_{1}=|e_{1}\alpha +\frac{1}{2}\alpha ^{2}+\frac{1}{6}H^{2}+\frac{1}{4}W-%
\func{Im}A_{11}|^{3/2}\Theta \wedge e^{1}
\end{equation*}%
is a (pointwise) CR invariant, i.e. under the conformal change of contact
form: $\tilde{\Theta}=\lambda ^{2}\Theta ,$ $\lambda >0,$ we have $d\tilde{A}%
_{1}=dA_{1}.$
\end{proposition}

\begin{proof}
(direct verification which is different from the original proof via the
Cartan method) The transformation laws of $e_{1},e_{2},T,$ $\alpha ,$ $H,$ $%
W $ and $A_{11}$ under the conformal change of contact form: $\tilde{\Theta}%
=\lambda ^{2}\Theta ,$ $\lambda >0,$ read as follows (on a domain of
nonsingular points): $i)$ $\tilde{e}_{1}=\lambda ^{-1}e_{1},$ $ii)$ $\tilde{e%
}_{2}=\lambda ^{-1}e_{2},$ $iii)$ $\tilde{T}=\lambda ^{-2}T+\lambda
^{-3}(e_{2}(\lambda )e_{1}-e_{1}(\lambda )e_{2}),$ $iv)$ $\tilde{\alpha}%
=\lambda ^{-1}\alpha +\lambda ^{-2}e_{1}(\lambda ),$ $v)$ $\tilde{H}=\lambda
^{-1}H-3\lambda ^{-2}e_{2}(\lambda ),$ $vi)$ $\tilde{W}=\lambda
^{-2}(W-4\lambda ^{-1}\Delta _{b}\lambda ),$ $vii)$ $\tilde{A}_{11}=\lambda
^{-2}(A_{11}+2i\lambda ^{-1}\lambda _{11}-6i\lambda ^{-2}\lambda _{1}\lambda
_{1}).$ A lengthy but straightforward computation yields%
\begin{equation*}
|\tilde{e}_{1}\tilde{\alpha}+\frac{1}{2}\tilde{\alpha}^{2}+\frac{1}{6}\tilde{%
H}^{2}+\frac{1}{4}\tilde{W}-\func{Im}\tilde{A}_{11}|=\lambda
^{-2}|e_{1}\alpha +\frac{1}{2}\alpha ^{2}+\frac{1}{6}H^{2}+\frac{1}{4}W-%
\func{Im}A_{11}|
\end{equation*}

\noindent while $\tilde{\Theta}\wedge \tilde{e}^{1}=\lambda ^{3}\Theta
\wedge e^{1}.$ That $d\tilde{A}_{1}=dA_{1}$ follows.
\end{proof}

\section{Proof of Theorem \protect\ref{mainthm2}: the local uniqueness\label%
{Sec3}}

For a nonsingular ($T\Sigma \neq \xi $) surface $\Sigma $ in a contact
3-manifold ($M,\xi ).$ We choose a characteristic tangent $e_{1}\in T\Sigma
\cap \xi ,$ unit with respect to the Levi-metric $\frac{1}{2}d\Theta (\cdot
,J\cdot )$ where $\Theta $ is the stardard contact form in $M=H_{1}.$ For
basic material in the theory of surfaces in the Heisenberg group, we refer
the reader to \cite{CHMY05} and \cite{CHMY12}. We denote by $N^{\perp }$
(resp. $N)$ the projection of $-e_{1}$ (resp. $-(e_{2}+\alpha ^{-1}T)$) onto
the domain of $u$ in the $xy$-plane. From the theory of hyperbolic
equations, we find out two \textquotedblleft characteristic" directions to
be $N^{\perp }$ and $N+\frac{1}{3}\alpha ^{-1}HN^{\perp }$: (see the
formulas for $F_{u_{xx}}$, $F_{u_{xy}},$ $F_{u_{yy}}$ in (\ref{A-6a}) and (%
\ref{A-6b}))%
\begin{equation*}
\begin{split}
& (N+\frac{1}{3}\alpha ^{-1}HN^{\perp })\circ N^{\perp } \\
=& F_{u_{xx}}\partial _{x}^{2}+F_{u_{xy}}\partial _{x}\partial
_{y}+F_{u_{yy}}\partial _{y}^{2}\ \ \text{mod lower order terms}.
\end{split}%
\end{equation*}%
The Codazzi-like equation (see, for instance, \cite[(3.24) with $W=0$]{CYZ18}%
) reads as 
\begin{equation}
-\alpha N(H)=N^{\perp }N^{\perp }\alpha -6\alpha N^{\perp }(\alpha )+4\alpha
^{3}+\alpha H^{2}.  \label{codeq1}
\end{equation}

We have the following identities (see (\ref{1-1}) for the notation $\theta $)%
\begin{equation}
N^{\perp }\theta =-H  \label{iden1}
\end{equation}%
\noindent (see, for instance, \cite[(1.14) with $V=N^{\perp }$]{CHMY12}, not
to be confused with the notation $V$ in \cite{CYZ18}) and 
\begin{equation}
N^{\perp }(\alpha )-\alpha N(\theta )-2\alpha ^{2}=0.  \label{iden2}
\end{equation}

\noindent (see, for instance, \cite[(1.15) with the notation $D$ replaced by 
$-\alpha ^{-1}$ and $V=N^{\perp }$]{CHMY12}).

\begin{proposition}
\label{P-1} For an absolute $E_1$-minimizer, we have 
\begin{equation}
\alpha (N+\frac{1}{3}\alpha ^{-1}HN^{\perp })\theta +N^{\perp }(\alpha
)=-\alpha ^{2}.  \label{emini1}
\end{equation}
\end{proposition}

\begin{proof}
First of all, we notice that, writing down in the domain of $u$, (\ref%
{enzeeq1}) reads as 
\begin{equation}
2N^{\perp }(\alpha )=\alpha ^{2}+\frac{1}{3}H^{2}.  \label{enzeeq15}
\end{equation}%
On the other hand, 
\begin{equation}
\begin{split}
\alpha (N+\frac{1}{3}\alpha ^{-1}HN^{\perp })\theta & =\alpha N\theta -\frac{%
1}{3}H^{2}\ \ (\text{by}\ (\ref{iden1})) \\
& =N^{\perp }\alpha -2\alpha ^{2}-\frac{1}{3}H^{2}\ \ (\text{by}\ (\ref%
{iden2})) \\
& =-N^{\perp }(\alpha )-\alpha ^{2}\ \ (\text{by}\ (\ref{enzeeq15})).
\end{split}%
\end{equation}%
This completes the proof.
\end{proof}

Later we will see that, in order to get a quasi-linear system, we need an
extra equation including the term $N^{\perp}H$, or $N^{\perp}N^{\perp}\theta$%
.

\begin{proposition}
\label{P-2} For an absolute $E_1$-minimizer, we have 
\begin{equation}
-\alpha (NH)-\alpha H^{2}+2\alpha ^{2}(N\theta )+\frac{2}{3}HN^{\perp
}N^{\perp }\theta +N^{\perp }N^{\perp }\alpha =-2\alpha (N^{\perp }\alpha ).
\label{emini2}
\end{equation}
\end{proposition}

\begin{proof}
We take $N^{\perp }$-derivative of (\ref{emini1}) to get 
\begin{equation}
\left( (N^{\perp }\alpha )N+\alpha N^{\perp }N+\frac{1}{3}(N^{\perp
}H)N^{\perp }+\frac{1}{3}HN^{\perp }N^{\perp }\right) \theta +N^{\perp
}N^{\perp }\alpha =-2\alpha (N^{\perp }\alpha ),  \label{demini1}
\end{equation}%
where 
\begin{equation}
\frac{1}{3}(N^{\perp }H)N^{\perp }\theta =\frac{1}{3}(-N^{\perp }N^{\perp
}\theta )(-H)=\frac{1}{3}HN^{\perp }N^{\perp }\theta .  \label{demini2}
\end{equation}%
Substituting (\ref{demini2}) into (\ref{demini1}), we get 
\begin{equation}
(N^{\perp }\alpha )N\theta +\alpha N^{\perp }N\theta +\frac{2}{3}HN^{\perp
}N^{\perp }\theta +N^{\perp }N^{\perp }\alpha =-2\alpha (N^{\perp }\alpha ),
\label{demini3}
\end{equation}%
where 
\begin{equation}
N^{\perp }N\theta =NN^{\perp }\theta +[N^{\perp },N]\theta =-NH+[N^{\perp
},N]\theta .  \label{demini4}
\end{equation}%
On the other hand, writing down the formula (see \cite[(3.42)]{CYZ18}) 
\begin{equation*}
\lbrack e_{1},V]=-\alpha He_{1}-2\alpha V,\ \ \text{where}\ V=\alpha e_{2}+T,
\end{equation*}%
in the domain of $u$, we get 
\begin{equation*}
\lbrack -N^{\perp },-\alpha N]=\alpha HN^{\perp }+2\alpha ^{2}N,
\end{equation*}%
that is, 
\begin{equation}
\alpha \lbrack N^{\perp },N]=\alpha HN^{\perp }+(2\alpha ^{2}-N^{\perp
}\alpha )N.  \label{demini5}
\end{equation}%
Substituting (\ref{demini5}) into (\ref{demini4}) and (\ref{demini4}) into (%
\ref{demini3}) respectively, we obtain (\ref{emini2}). This completes the
proof.
\end{proof}

We take all equations (\ref{codeq1}), (\ref{iden2}), (\ref{emini1}) and (\ref%
{emini2}) together to obtain a system as follows 
\begin{equation}
\left( 
\begin{array}{cccc}
\alpha (N+\frac{1}{3}\alpha ^{-1}HN^{\perp }) & N^{\perp } & 0 & 0 \\ 
-\alpha N & N^{\perp } & 0 & 0 \\ 
0 & 6\alpha N^{\perp } & -\alpha N & N^{\perp } \\ 
2\alpha ^{2}N & 2\alpha N^{\perp } & -\alpha N-\frac{2}{3}HN^{\perp } & 
-N^{\perp }%
\end{array}%
\right) \left( 
\begin{array}{c}
\theta \\ 
\alpha \\ 
H \\ 
-N^{\perp }\alpha%
\end{array}%
\right) =\left( 
\begin{array}{c}
-\alpha ^{2} \\ 
2\alpha ^{2} \\ 
4\alpha ^{3}+\alpha H^{2} \\ 
\alpha H^{2}%
\end{array}%
\right)  \label{enzeeq2}
\end{equation}%
Notice that (\ref{enzeeq2}) only holds on the region where $\alpha \neq 0$.
In terms of polar coordinates $(r,\phi )$, we have (see also (\ref{fifuda})) 
\begin{equation}
\begin{split}
N^{\perp }& =\sin {(\theta -\phi )}\partial _{r}-\frac{\cos {(\theta -\phi )}%
}{r}\partial _{\phi } \\
N& =\cos {(\theta -\phi )}\partial _{r}+\frac{\sin {(\theta -\phi )}}{r}%
\partial _{\phi }.
\end{split}
\label{fifuda}
\end{equation}%
Writing $U$ as the column vector 
\begin{equation}  \label{U}
U=\left( 
\begin{array}{c}
\theta \\ 
\alpha \\ 
H \\ 
-N^{\perp }\alpha%
\end{array}%
\right) ,
\end{equation}%
and $s=\sin {(\theta -\phi )},\ c=\cos {(\theta -\phi )}$, then (\ref%
{enzeeq2}) reads as

\begin{equation}  \label{enzeeq3}
A(r,\phi,U)\frac{\partial U}{\partial r}+B(r,\phi,U)\frac{\partial U}{%
\partial \phi}+C(r,\phi,U)U=0,
\end{equation}
where 
\begin{equation}  \label{Aex}
A(r,\phi,U)=\left(%
\begin{array}{cccc}
\frac{1}{3}sH+c\alpha & s & 0 & 0 \\ 
-c\alpha & s & 0 & 0 \\ 
0 & 6s\alpha & -c\alpha & s \\ 
2c\alpha^{2} & 2s\alpha & -\frac{2}{3}sH-c\alpha & -s%
\end{array}%
\right)
\end{equation}
\begin{equation}  \label{Bex}
B(r,\phi,U)=\left(%
\begin{array}{cccc}
-\frac{1}{3}\frac{c}{r}H+\frac{s}{r}\alpha & -\frac{c}{r} & 0 & 0 \\ 
-\frac{s}{r}\alpha & -\frac{c}{r} & 0 & 0 \\ 
0 & -6\frac{c}{r}\alpha & -\frac{s}{r}\alpha & -\frac{c}{r} \\ 
2\frac{s}{r}\alpha^{2} & -2\frac{c}{r}\alpha & \frac{2}{3}\frac{c}{r}H-\frac{%
s}{r}\alpha & \frac{c}{r}%
\end{array}%
\right),
\end{equation}

\begin{equation}
C(r,\phi ,U)=(-1)\left( 
\begin{array}{cccc}
0 & -\alpha & 0 & 0 \\ 
0 & 2\alpha & 0 & 0 \\ 
0 & 4\alpha ^{2} & \alpha H & 0 \\ 
0 & 0 & \alpha H & 0%
\end{array}%
\right) .  \label{Cex}
\end{equation}%
Equation (\ref{enzeeq3}) is a quasi-linear equation of first order. The
Cauchy problem for (\ref{enzeeq3}) prescribes the value of $U$ on a curve $%
r=r(\phi )$ in the $\phi r$-plane 
\begin{equation}
U=U(r(\phi ),\phi )=f(\phi ).  \label{Inval}
\end{equation}%
The curve is characteristic \cite[p.47]{John82} if 
\begin{equation}
\det {(Ad\phi -Bdr)}=0  \label{chaeq}
\end{equation}%
for each point $(r(\phi ),\phi )$ on the curve. It depends on the solution $%
U $. On the region where $\det {A(r,\phi ,U)}\neq 0$, the quasi-linear
equation (\ref{enzeeq3}) becomes 
\begin{equation}
U_{r}+a(r,\phi ,U)U_{\phi }+b(r,\phi ,U)=0,  \label{enzeeq4}
\end{equation}%
where $a=A^{-1}B$ and $b=A^{-1}CU$. On this region, we obtain immediately
from the characteristic differential equation (\ref{chaeq}) that the curve $%
r=\mathbf{c}>0$, where $\frac{dr}{d\phi}=0$, is always non-characteristic.
The initial-value problem to be studied here is to prescribe the values of $U
$ on the curve $r=\mathbf{c}>0$, which is assumed to be non-characteristic: $%
U(\mathbf{c},\phi )=f(\phi ) $.

\subsection{The smooth category}

In this subsection, we show the uniqueness theorem for the initial-value
problem.

\begin{theorem}
\label{mainthm1} Let $U$ be a solutions to (\ref{enzeeq3}) or to (\ref%
{enzeeq4}) and the curve $r=\mathbf{c}$ is lying in the region where $\det {%
A(r,\phi ,U)}\neq 0$. If $V$ is another solution to (\ref{enzeeq3}) or to (%
\ref{enzeeq4}) having the same initial values as $U$, on the curve $r=%
\mathbf{c}>0$. Then $U=V$ on a neighborhood of the initial curve $r=\mathbf{c%
}$.
\end{theorem}

\begin{proof}
Notice that 
\begin{equation}
\det {A(r,\phi ,U)}=s^{2}(\frac{1}{3}sH+2c\alpha )(\frac{2}{3}sH+2c\alpha ).
\label{deoA}
\end{equation}%
Since $\det {A(r,\phi ,U)}\neq 0$, and $V$ and $U$ have the same initial
values on the curve $r=\mathbf{c}$, we conclude that $\det {A(r,\phi ,V)}%
\neq 0$ on a neighborhood of the initial curve $r=\mathbf{c}$. We consider
the vector function 
\begin{equation*}
Z(r,\phi ):=V-U.
\end{equation*}%
Then $Z$ has the initial values zero on the curve $r=\mathbf{c}$.

Subtracting the differential equation (\ref{enzeeq4}) for $U$ from that for $%
V$ we have, omitting explicit reference to $r$ and $\phi $,

\begin{equation}  \label{enzeeq5}
Z_{r}+a(U)Z_{\phi}+[a(V)-a(U)]V_{\phi}+b(V)-b(U)=0.
\end{equation}

On the other hand, we may apply the mean value theorem and have 
\begin{equation*}
a(V)-a(U)=h(V,U)Z;\ \ \ b(V)-b(U)=k(V,U)Z,
\end{equation*}%
for some $h$ and $k$ which are both smooth. We now consider $U,V,V_{\phi }$
as known expressions in $r,\phi $ and substitute these expressions in $h$
and $k$ as well as in $a(U)$; thereby (\ref{enzeeq5}) becomes a linear
differential equation for $Z$ of the form 
\begin{equation}
Z_{r}+a(r,\phi )Z_{\phi }+b(r,\phi )Z=0,  \label{enzeeq6}
\end{equation}%
with initial values zero. Here $a(r,\phi )=a(U)=a(r,\phi ,U)$. In the next
subsection, we will show that the system (\ref{enzeeq6}) is hyperbolic (see 
\cite[p.48]{John82}) in the sense that there exists a complete set of real
eigenvectors $\xi ^{1},\xi ^{2},\xi ^{3},\xi ^{4}$ of $a(r,\phi )$ by
Proposition \ref{P-4}, such that 
\begin{equation*}
a(r,\phi )\xi ^{k}=\lambda _{k}\xi ^{k},
\end{equation*}%
where the $\xi ^{k}$ are linearly independent and depend regularly on $r$
and $\phi $. The uniqueness theorem for a first-order system of linear
equations (\cite[Section 5 in Chapter 2]{John82}) asserts now that $Z$ is
identically zero, and thus $U=V$ around the curve $r=\mathbf{c}$.
\end{proof}

\begin{proof}
\textbf{(of Theorem \ref{mainthm2})} From (\ref{fifuda}), we see that the
same value for $\theta $ via Theorem \ref{mainthm1} implies that the two
graphs $u=u(r,\phi )$ and $v=v(r,\phi )$ have the same first fundamental
form, as well as the characteristic vectors $N^{\perp }$, (or $e_{1}$).
Together with the same value for $\alpha $ and $H$ by Theorem \ref{mainthm1}%
, we conclude, by the fundamental theorem for surfaces in $H_{1}$ (for the
details, see \textbf{Theorem 1.3} in \cite{CL}, also cf. \cite[Theorem H]%
{CHMY12}), that $u=v$ on a neighborhood of the initial curve $r=\mathbf{c}$.
\end{proof}

\subsection{Hyperbolic system}

In this subsection we are going to show that the system (\ref{enzeeq6}) is
hyperbolic in the sense of \cite[p.48]{John82} so that Theorem \ref{mainthm1}%
, and thus Theorem \ref{mainthm2}, hold. Let 
\begin{equation}
\sigma =\frac{1}{3}sH+2c\alpha ,\ \ \ \eta =\frac{2}{3}sH+2c\alpha ,
\end{equation}%
that is, from (\ref{deoA}), $\det {A}(r,\phi ,U)=s^{2}\sigma \eta $. Then,
applying a sequence of elementary row operations, we obtain that the inverse
of $A(r,\phi ,U)$ is 
\begin{equation}
A^{-1}(r,\phi ,U)=\left( 
\begin{array}{cccc}
\frac{1}{\sigma } & -\frac{1}{\sigma } & 0 & 0 \\ 
\frac{c\alpha }{s\sigma } & \frac{\sigma -c\alpha }{s\sigma } & 0 & 0 \\ 
\frac{10c\alpha ^{2}}{\sigma \eta } & \frac{8\sigma \alpha -10c\alpha ^{2}}{%
\sigma \eta } & -\frac{1}{\eta } & -\frac{1}{\eta } \\ 
\frac{-6\eta c\alpha ^{2}+10c^{2}\alpha ^{3}}{s\sigma \eta } & \frac{8\sigma
c\alpha ^{2}-10c^{2}\alpha ^{3}-6\sigma \eta \alpha +6\eta c\alpha ^{2}}{%
s\sigma \eta } & \frac{\eta -c\alpha }{s\eta } & -\frac{c\alpha }{s\eta }%
\end{array}%
\right) ,  \label{eoAin}
\end{equation}%
and thus 
\begin{equation}
a(r,\phi )=A^{-1}B(r,\phi ,U)=\left( 
\begin{array}{cccc}
\frac{-\frac{1}{3}cH+2s\alpha }{r\sigma } & 0 & 0 & 0 \\ 
\frac{-\frac{1}{3}\alpha H}{sr\sigma } & \frac{-c}{sr} & 0 & 0 \\ 
\frac{-\frac{10}{3}H\alpha ^{2}}{r\sigma \eta } & 0 & \frac{-\frac{2}{3}%
cH+2s\alpha }{r\eta } & 0 \\ 
\frac{\frac{4}{3}s\alpha ^{2}H^{2}+\frac{2}{3}c\alpha ^{3}H}{sr\sigma \eta }
& 0 & \frac{-\frac{2}{3}\alpha H}{sr\eta } & \frac{-c}{sr}%
\end{array}%
\right) ,  \label{eoAinB}
\end{equation}%
with eigenvalues 
\begin{equation}
\lambda _{1}=\frac{-\frac{1}{3}cH+2s\alpha }{r\sigma },\ \ \ \lambda
_{2}=\lambda _{4}=\frac{-c}{sr},\ \ \ \lambda _{3}=\frac{-\frac{2}{3}%
cH+2s\alpha }{r\eta }.
\end{equation}

\begin{proposition}
\label{P-3} We have that $\lambda_{1}\neq 0\ \ \lambda_{3}\neq 0$. And

\begin{enumerate}
\item If $\alpha H\neq 0$, then $\lambda_{1}\neq\lambda_{2},\
\lambda_{2}\neq\lambda_{3},\ \lambda_{1}\neq\lambda_{3}$.

\item If $\alpha\neq 0,\ H=0$, then $\lambda_{1}\neq\lambda_{2},\
\lambda_{2}\neq\lambda_{3},\ \lambda_{1}=\lambda_{3}$.

\item If $\alpha=0,\ H\neq 0$, then $\lambda_{1}=\lambda_{2}=\lambda_{3}=%
\lambda_{4}$.
\end{enumerate}
\end{proposition}

\begin{proof}
Let 
\begin{equation}
\tilde{\sigma}=-\frac{1}{3}cH+2s\alpha,\ \ \ \tilde{\eta}=-\frac{2}{3}%
cH+2s\alpha,
\end{equation}
it is easy to see that 
\begin{equation}
\tilde{\sigma}(r,\phi+\frac{\pi}{2},U)=-\sigma(r,\phi,U),\ \ \ \tilde{\eta}%
(r,\phi+\frac{\pi}{2},U)=-\eta(r,\phi,U),
\end{equation}
thus, on the region where $\det{A(r,\phi,U)}\neq 0$, we have that $%
\lambda_{1}\neq 0$ and $\lambda_{3}\neq 0$. On the other hand, a
straightforward computation shows that 
\begin{equation}
\lambda_{1}-\lambda_{2}=\frac{2\alpha}{sr\sigma},\ \ \
\lambda_{3}-\lambda_{2}=\frac{2\alpha}{sr\eta},\ \ \ \lambda_{1}-\lambda_{3}=%
\frac{\frac{2}{3}\alpha H}{r\sigma\eta}.
\end{equation}
This completes the proof of the proposition.
\end{proof}

We have 
%

\begin{proposition}
\label{P-4} The eigenspaces with respect to the eigenvalues $\lambda
_{1},\lambda _{2}(=\lambda _{4}),\lambda _{3}$ respectively are spanned by
the vectors 
\begin{equation*}
\xi ^{1}=\left( 
\begin{array}{c}
1 \\ 
-\frac{1}{6}H \\ 
-5\alpha \\ 
\frac{11}{6}\alpha H%
\end{array}%
\right) ,\ \ \xi ^{2}=\left( 
\begin{array}{c}
0 \\ 
1 \\ 
0 \\ 
0%
\end{array}%
\right) ,\ \ \xi ^{4}=\left( 
\begin{array}{c}
0 \\ 
0 \\ 
0 \\ 
1%
\end{array}%
\right) ,\ \ \xi ^{3}=\left( 
\begin{array}{c}
0 \\ 
0 \\ 
1 \\ 
-\frac{1}{3}H%
\end{array}%
\right) .
\end{equation*}%
It is clear that the $\xi ^{k},k=1\cdots 4$ are linearly independent and
depend regularly on $r$ and $\phi $. This means that the system (\ref%
{enzeeq6}) is hyperbolic.
\end{proposition}

\begin{proof}
We in addition put 
\begin{equation}
\lambda _{21}=\frac{-\frac{1}{3}\alpha H}{sr\sigma },\ \ \ \lambda _{31}=%
\frac{-\frac{10}{3}H\alpha ^{2}}{r\sigma \eta },\ \ \ \lambda _{41}=\frac{%
\frac{4}{3}s\alpha ^{2}H^{2}+\frac{2}{3}c\alpha ^{3}H}{sr\sigma \eta },\ \ \
\lambda _{43}=\frac{-\frac{2}{3}\alpha H}{sr\eta },
\end{equation}%
which are the other nonzero entries in $a(r,\phi )$ in (\ref{eoAinB}).
First, we suppose that $\alpha H\neq 0$. Assume that $%
\xi=(a_{1},a_{2},a_{3},a_{4})^{t}$ is an eigenvector of $a(r,\phi)$ with
respect to $\lambda_{1}$, then $a(r,\phi)\xi=\lambda_{1}\xi$ means that 
\begin{equation*}
\begin{split}
\lambda _{1}a_{1}& =\lambda _{1}a_{1} \\
\lambda _{21}a_{1}+\lambda _{2}a_{2}& =\lambda _{1}a_{2} \\
\lambda _{31}a_{1}+\lambda _{3}a_{3}& =\lambda _{1}a_{3} \\
\lambda _{41}a_{1}+\lambda _{43}a_{3}+\lambda _{4}a_{4}& =\lambda _{1}a_{4},
\end{split}%
\end{equation*}%
which implies that 
\begin{equation*}
\xi =a_{1}\left( 
\begin{array}{c}
1 \\ 
\frac{\lambda _{21}}{\lambda _{1}-\lambda _{2}} \\ 
\frac{\lambda _{31}}{\lambda _{1}-\lambda _{3}} \\ 
\frac{\lambda _{41}}{\lambda _{1}-\lambda _{4}}+\frac{\lambda _{31}\lambda
_{43}}{(\lambda _{1}-\lambda _{3})(\lambda _{1}-\lambda _{4})}%
\end{array}%
\right) ,
\end{equation*}%
for $a_{1}$ arbitrary. After a straightforward computation, and noticing
that $\alpha H\neq 0$, we have 
\begin{equation*}
\xi ^{1}=\left( 
\begin{array}{c}
1 \\ 
\frac{\lambda _{21}}{\lambda _{1}-\lambda _{2}} \\ 
\frac{\lambda _{31}}{\lambda _{1}-\lambda _{3}} \\ 
\frac{\lambda _{41}}{\lambda _{1}-\lambda _{4}}+\frac{\lambda _{31}\lambda
_{43}}{(\lambda _{1}-\lambda _{3})(\lambda _{1}-\lambda _{4})}%
\end{array}%
\right) =\left( 
\begin{array}{c}
1 \\ 
-\frac{1}{6}H \\ 
-5\alpha \\ 
\frac{11}{6}\alpha H%
\end{array}%
\right) .
\end{equation*}%
Similarly, since $\lambda _{1}\neq \lambda _{2},\ \lambda _{3}\neq \lambda
_{2}$ and $\lambda _{1}\neq \lambda _{3}$, it is easy to get that 
\begin{equation*}
\xi ^{2}=\left( 
\begin{array}{c}
0 \\ 
1 \\ 
0 \\ 
0%
\end{array}%
\right) ,\ \ \xi ^{4}=\left( 
\begin{array}{c}
0 \\ 
0 \\ 
0 \\ 
1%
\end{array}%
\right) ,\ \ \xi ^{3}=\left( 
\begin{array}{c}
0 \\ 
0 \\ 
1 \\ 
-\frac{1}{3}H%
\end{array}%
\right) .
\end{equation*}%
Next we assume that $\alpha\neq 0,\ H=0$, then $\lambda_{1}=\lambda_{3}$,
and a similar computation shows that the eigensapce with respect to $%
\lambda_{1} (\text{or}\ \lambda_{3})$ is spanned by the two vectors 
\begin{equation*}
\xi^{1}=\left(%
\begin{array}{c}
1 \\ 
0 \\ 
-5\alpha \\ 
0%
\end{array}%
\right),\ \ \xi^{3}=\left(%
\begin{array}{c}
0 \\ 
0 \\ 
1 \\ 
0%
\end{array}%
\right).
\end{equation*}
And the eigensapce with respect to $\lambda_{2}(\text{or}\ \lambda_{4})$ is
spanned by the two vectors 
\begin{equation*}
\xi^{2}=\left(%
\begin{array}{c}
0 \\ 
1 \\ 
0 \\ 
0%
\end{array}%
\right),\ \ \xi^{4}=\left(%
\begin{array}{c}
0 \\ 
0 \\ 
0 \\ 
1%
\end{array}%
\right).
\end{equation*}
Finally, we assume that $\alpha=0,\ H\neq 0$, then $\lambda_{1}=\lambda_{2}=%
\lambda_{3}=\lambda_{4}$, that is, $a(r,\phi)$ is a multiple of the identity
matrix of order $4$. This implies that each nonzero vector is an eigenvector
of $a(r,\phi)$. This completes the proof.
\end{proof}

\subsection{The real analytic category}

In the real analytic category, $r\in \R,\ \phi \in \R$ and $U\in \R^{4}$,
both $a$ and $b$ in (\ref{enzeeq4}) can be regarded as functions defined on $%
\R^{1}\times \R^{1}\times \R^{4}$, mapping into $R^{4\times 4}$ and $R^{4}$,
respectively. Looking at formulae (\ref{eoAin}), (\ref{Aex}), (\ref{Bex})
and (\ref{Cex}), we see that both $a$ and $b$ are clearly analytic. If we in
addition assume that the initial value function $f(\phi )=U(\mathbf{c},\phi
) $ is analytic, then by Cauchy-Kovalevskaya theorem, (\ref{enzeeq3}) or (%
\ref{enzeeq4}) has a unique solution on a neighborhood of $r=\mathbf{c}$
with initial value $f(\phi )$.

\section{Proof of Theorem \protect\ref{T-class}: the classification\label%
{Sec4c}}

\subsection{Rotationally symmetric surfaces}

Suppose in addition that it is rotationally symmetric, that is, $%
t=u(x,y)=u(r)$, which only depends on $r,r=(x^{2}+y^{2})^{1/2}$. Then (\ref%
{enzeeq7}) is reduced to a non-linear ODE of second order

\begin{equation}  \label{enzeeq8}
\frac{1}{3}r^4u_{rr}^{2}-(\frac{4}{3}ru_{r}^{3}+2r^{3}u_{r})u_{rr}+\frac{1}{3%
}\frac{u_{r}^{6}}{r^2}+u_{r}^{4}-r^{4}=0.
\end{equation}

We observe that $u$ is a solution to (\ref{enzeeq8}) if and only if $-u$ is
a solution to (\ref{enzeeq8}).

\begin{proposition}
\label{bapr1} The function $u=u(r)$ satisfies equation (\ref{enzeeq8}) if
and only if it satisfies either%
\begin{equation}
u_{rr}=\frac{2u_{r}^{3}+3r^{2}u_{r}+\sqrt{3}(u_{r}^{2}+r^{2})^{3/2}}{r^{3}}
\label{enzeeq9}
\end{equation}%
or 
\begin{equation}
u_{rr}=\frac{2u_{r}^{3}+3r^{2}u_{r}-\sqrt{3}(u_{r}^{2}+r^{2})^{3/2}}{r^{3}}
\label{enzeeq10}
\end{equation}
\end{proposition}

\begin{proof}
Let 
\begin{equation*}
A(r)=\frac{1}{3}r^{4},\ \ B(r,u_{r})=-(\frac{4}{3}ru_{r}^{3}+2r^{3}u_{r}),\
\ C(r,u_{r})=\frac{1}{3}\frac{u_{r}^{6}}{r^{2}}+u_{r}^{4}-r^{4}.
\end{equation*}%
Then (\ref{enzeeq8}) can be written as 
\begin{equation*}
A(r)u_{rr}^{2}+B(r,u_{r})u_{rr}+C(r,u_{r})=0.
\end{equation*}%
Thus we immediately have 
\begin{equation}
u_{rr}=\frac{-B(r,u_{r})\pm \sqrt{B^{2}(r,u_{r})-4A(r)C(r,u_{r})}}{2A(r)},
\label{enzeeq11}
\end{equation}%
where 
\begin{equation*}
\begin{split}
B^{2}(r,u_{r})-4A(r)C(r,u_{r})& =\frac{4}{3}%
r^{2}u_{r}^{6}+4r^{4}u_{r}^{4}+4r^{6}u_{r}^{2}+\frac{4}{3}r^{8} \\
& =\frac{4}{3}r^{2}(u_{r}^{2}+r^{2})^{3}\geq 0.
\end{split}%
\end{equation*}%
This completes the proof.
\end{proof}

\begin{proposition}
\label{bapr2} The function $u=u(r)(\text{or}-u(r))$ satisfies equation (\ref%
{enzeeq9}) if and only if $-u(\text{or}\ u(r))$ satisfies equation (\ref%
{enzeeq10}).
\end{proposition}

\begin{proof}
First notice that $D^{2}=u_{r}^{2}+r^{2}$. To reach a contradiction, we
assume that both $u$ and $-u$ satisfy (\ref{enzeeq9}). Then 
\begin{equation*}
\begin{split}
0& =u_{rr}+(-u)_{rr} \\
& =\frac{2u_{r}^{3}+3r^{2}u_{r}+\sqrt{3}(u_{r}^{2}+r^{2})^{3/2}}{r^{3}}+%
\frac{-2u_{r}^{3}-3r^{2}u_{r}+\sqrt{3}(u_{r}^{2}+r^{2})^{3/2}}{r^{3}} \\
& =\frac{2\sqrt{3}(u_{r}^{2}+r^{2})^{3/2}}{r^{3}},
\end{split}%
\end{equation*}%
which is a contradiction on the regular part in which $D>0$. Similarly, if
we assume that both $u$ and $-u$ satisfy (\ref{enzeeq10}), we will also get
a contradiction. This completes the proof.
\end{proof}

In view of Proposition \ref{bapr1} and Proposition \ref{bapr2}, we see that,
to study equation (\ref{enzeeq8}), it is sufficient to study equation (\ref%
{enzeeq10}). Since (\ref{enzeeq10}) is independent of $u$, essentially it is
an ODE of first order for $u_{r}$. That is, if we let $v=u_{r}$ then (\ref%
{enzeeq10}) is just the first order equation for $v$ 
\begin{equation}
v_{r}=\frac{2v^{3}+3r^{2}v-\sqrt{3}(v^{2}+r^{2})^{3/2}}{r^{3}}.
\label{enzeeq12}
\end{equation}%
Moreover if we put $w=\frac{v}{r}$ then (\ref{enzeeq12}) can be expressed as
a little more simple form 
\begin{equation}
w_{r}=\frac{2w^{3}+2w-\sqrt{3}(1+w^{2})^{3/2}}{r},\ \ r>0.  \label{enzeeq13}
\end{equation}%
which can be solved in a closed form. We observe the numerator of the right
hand side of (\ref{enzeeq13}) and notice that $w_{r}\geq 0$ if and only if $%
2w^{3}+2w-\sqrt{3}(1+w^{2})^{3/2}\geq 0$, which is equivalent to that $w\geq 
\sqrt{3}$. And it is easy to see that $w=\sqrt{3}$ is a solution to (\ref%
{enzeeq13}). Actually we have Theorem \ref{gesol} for the general solutions
as follows. See Figure 1 for the graph of $w=w(r).$

\begin{figure}[h]
\includegraphics[width=10.4cm]{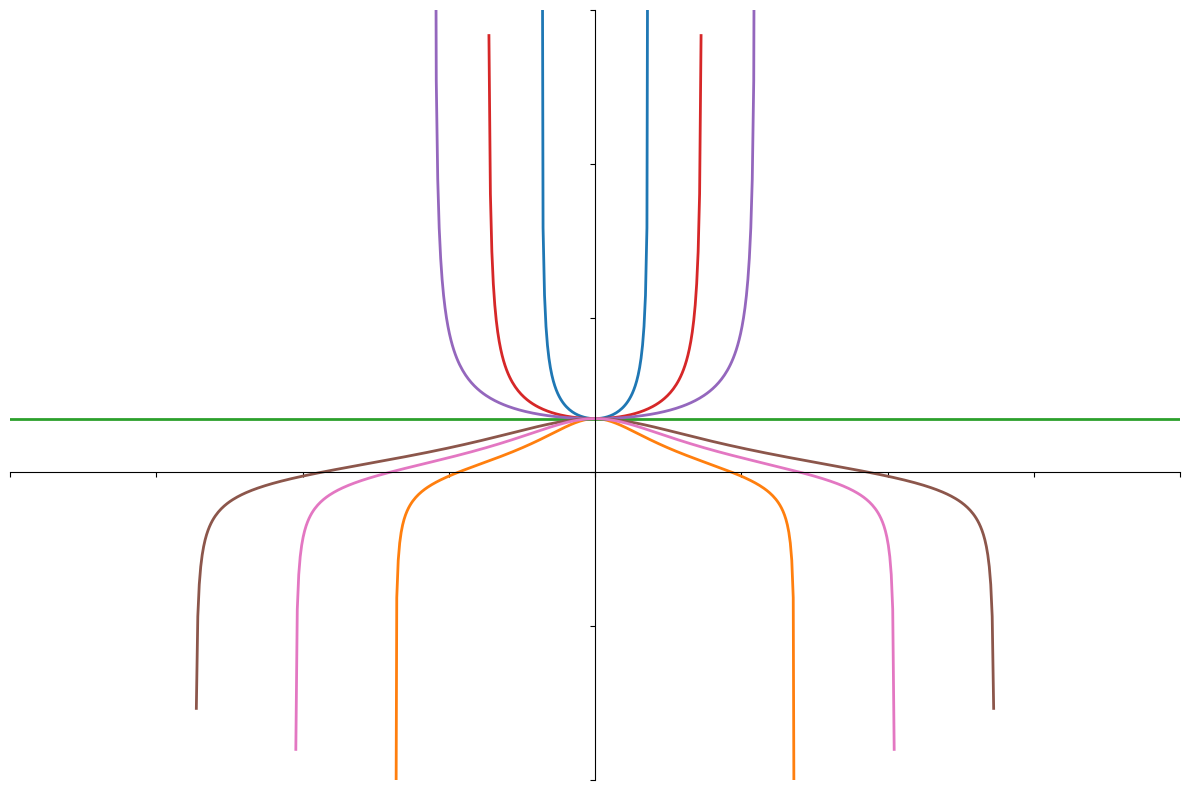}
\caption{(\protect\ref{soleq131}) and (\protect\ref{soleq132})}
\end{figure}

\begin{theorem}
\label{gesol} The general solutions to (\ref{enzeeq13}) is as follows:

\begin{enumerate}
\item The constant function $w=\sqrt{3}$ is a solution.

\item For each $\rho _{0}>0$, the function 
\begin{equation}
w(r)=\frac{r^{2}+\frac{\sqrt{3}}{2}\rho _{0}^{2}}{\sqrt{\rho _{0}^{4}-(r^{2}+%
\frac{\sqrt{3}}{2}\rho _{0}^{2})^{2}}},\ \ r>0,  \label{soleq131}
\end{equation}%
is a solution. We have that $w(r)>\sqrt{3}$ for all $r>0$. And 
\begin{equation*}
w(r)\rightarrow \left\{ 
\begin{array}{rl}
\sqrt{3}, & \text{as}\ r\rightarrow 0, \\ 
\infty , & \text{as}\ r\rightarrow r_{0}=\sqrt{\frac{2-\sqrt{3}}{2}}\rho
_{0}.%
\end{array}%
\right.
\end{equation*}%
Moreover, fixing a number $r$, we have $w(r)\rightarrow \sqrt{3}$ as $\rho
_{0}\rightarrow \infty $.

\item For each $a>0\ (or\ \rho _{0}>0,a^{2}=\frac{\sqrt{3}}{2}\rho _{0}^{2})$%
, the function 
\begin{equation}
w(r)=\frac{-\sqrt{3}(r^{2}-a^{2})}{\sqrt{4a^{4}-3(r^{2}-a^{2})^{2}}}=\frac{%
-(r^{2}-\frac{\sqrt{3}}{2}\rho _{0}^{2})}{\sqrt{\rho _{0}^{4}-(r^{2}-\frac{%
\sqrt{3}}{2}\rho _{0}^{2})^{2}}},\ \ r>0,  \label{soleq132}
\end{equation}%
is a solution with $w(a)=0$. We have that $w(r)<\sqrt{3}$ for all $r>0$, and 
\begin{equation*}
w(r)\rightarrow \left\{ 
\begin{array}{rl}
\sqrt{3}, & \text{as}\ r\rightarrow 0, \\ 
-\infty , & \text{as}\ r\rightarrow b=\sqrt{\frac{2+\sqrt{3}}{\sqrt{3}}}a=%
\sqrt{\frac{2+\sqrt{3}}{2}}\rho _{0}.%
\end{array}%
\right.
\end{equation*}%
Moreover, fixing a number $r$, we have $w(r)\rightarrow \sqrt{3}$ as $%
a\rightarrow \infty $.
\end{enumerate}
\end{theorem}

\begin{proof}
We will only show formulae (\ref{soleq131}) and (\ref{soleq132}) for $w(r)$
while the others follow easily. As equation (\ref{enzeeq13}) is separable,
taking the substitution $w=\tan {\theta },\ \theta \in (\frac{-\pi }{2},%
\frac{\pi }{3})$, we have $w\in (-\infty ,\sqrt{3})$ and 
\begin{equation*}
\begin{split}
\int \frac{dr}{r}& =\int \frac{dw}{2w^{3}+2w-\sqrt{3}(1+w^{2})^{3/2}} \\
& =\int \frac{\sec ^{2}{\theta }\ d\theta }{\sec ^{2}{\theta }(2\tan {\theta 
}-\sqrt{3}(\sec ^{2}{\theta })^{1/2})} \\
& =\int \frac{d\theta }{2\tan {\theta }-\sqrt{3}\sec {\theta }},\ \ \sec {%
\theta }>0 \\
& =\int \frac{\cos {\theta }\ d\theta }{2\sin {\theta }-\sqrt{3}},\ \ \cos {%
\theta }>0 \\
& =\int \frac{d\sin {\theta }}{2\sin {\theta }-\sqrt{3}} \\
& =\frac{1}{2}\ln {(\sqrt{3}-2\sin {\theta })},\ \ \text{where}\ \sqrt{3}%
-2\sin {\theta }>0 \\
& =\ln {\left( \sqrt{3}-\frac{2w}{\sqrt{1+w^{2}}}\right) ^{1/2}}.
\end{split}%
\end{equation*}%
Therefore 
\begin{equation}
\left( \sqrt{3}-\frac{2w}{\sqrt{1+w^{2}}}\right) ^{1/2}=Cr,  \label{sol1}
\end{equation}%
for some constant $C>0$. If we take $C=\left( \frac{\sqrt{3}}{a^{2}}\right)
^{1/2},\ a>0$, the solution (\ref{sol1}) reads as 
\begin{equation*}
\frac{4w^{2}}{1+w^{2}}=\frac{3}{a^{4}}(a^{2}-r^{2})^{2}
\end{equation*}%
or 
\begin{equation*}
w=\frac{-\sqrt{3}(r^{2}-a^{2})}{\sqrt{4a^{4}-3(r^{2}-a^{2})^{2}}}=\frac{%
-(r^{2}-\frac{\sqrt{3}}{2}\rho _{0}^{2})}{\sqrt{\rho _{0}^{4}-(r^{2}-\frac{%
\sqrt{3}}{2}\rho _{0}^{2})^{2}}}.
\end{equation*}%
This is the formula (\ref{soleq132}). Similarly, for $w>\sqrt{3}$, we can
show that 
\begin{equation*}
\int \frac{dr}{r}=\ln {\left( \frac{2w}{\sqrt{1+w^{2}}}-\sqrt{3}\right)
^{1/2}}
\end{equation*}%
Therefore 
\begin{equation}
\left( \frac{2w}{\sqrt{1+w^{2}}}-\sqrt{3}\right) ^{1/2}=Cr,  \label{sol2}
\end{equation}%
for some constant $C>0$. If we take $C=\left( \frac{2}{\rho _{0}^{2}}\right)
^{1/2}$, the solution (\ref{sol2}) reads as 
\begin{equation*}
w=\frac{r^{2}+\frac{\sqrt{3}}{2}\rho _{0}^{2}}{\sqrt{\rho _{0}^{4}-(r^{2}+%
\frac{\sqrt{3}}{2}\rho _{0}^{2})^{2}}}.
\end{equation*}%
This is the formula (\ref{soleq131}). We thus complete the proof.
\end{proof}

\begin{proof}
\textbf{(of Theorem \ref{T-class})} For solutions $w(r)$ to (\ref{enzeeq13}%
), the corresponding solutions $u(r)$ to (\ref{enzeeq10}) are, respectively,
computed as follows. If $w(r)=\sqrt{3}$, then 
\begin{equation}
\begin{split}
u(r)-u(0)& =\int_{0}^{r}u_{s}\ ds=\int_{0}^{r}sw(s)\ ds \\
& =\int_{0}^{r}\sqrt{3}s\ ds=\frac{\sqrt{3}}{2}r^{2}.
\end{split}%
\end{equation}%
So if $u(0)=0$ then we have $u(r)=\frac{\sqrt{3}}{2}r^{2}$. By Proposition %
\ref{bapr2}, $u(r)=-\frac{\sqrt{3}}{2}r^{2}$ is another solution. If $w(r)$
is one of (\ref{soleq131}), then 
\begin{equation}
\begin{split}
u(r)-u(0)& =\int_{0}^{r}u_{s}\ ds=\int_{0}^{r}sw(s)\ ds \\
& =\int_{0}^{r}s\frac{s^{2}+\frac{\sqrt{3}}{2}\rho _{0}^{2}}{\sqrt{\rho
_{0}^{4}-(s^{2}+\frac{\sqrt{3}}{2}\rho _{0}^{2})^{2}}}\ ds \\
& =\frac{1}{2}\int_{\frac{\sqrt{3}}{2}\rho _{0}^{2}}^{r^{2}+\frac{\sqrt{3}}{2%
}\rho _{0}^{2}}\frac{x}{\sqrt{\rho _{0}^{4}-x^{2}}}\ dx,\ \ x=s^{2}+\frac{%
\sqrt{3}}{2}\rho _{0}^{2} \\
& =\frac{1}{4}\rho _{0}^{2}-\frac{1}{2}\sqrt{\rho _{0}^{4}-\left( r^{2}+%
\frac{\sqrt{3}}{2}\rho _{0}^{2}\right) ^{2}}.
\end{split}%
\end{equation}%
Therefore if $u(0)=-\frac{1}{4}\rho _{0}^{2}$ then we have 
\begin{equation}
u(r)=-\frac{1}{2}\sqrt{\rho _{0}^{4}-\left( r^{2}+\frac{\sqrt{3}}{2}\rho
_{0}^{2}\right) ^{2}},\ \ 0<r<\sqrt{\frac{2-\sqrt{3}}{2}}\rho _{0},
\label{gra1}
\end{equation}%
which is part of a type I shifted Heisenberg sphere. If $w(r)$ is one of (%
\ref{soleq132}), then 
\begin{equation}
\begin{split}
u(r)-u(0)& =\int_{0}^{r}u_{s}\ ds=\int_{0}^{r}sw(s)\ ds \\
& =-\int_{0}^{r}s\frac{s^{2}-\frac{\sqrt{3}}{2}\rho _{0}^{2}}{\sqrt{\rho
_{0}^{4}-(s^{2}-\frac{\sqrt{3}}{2}\rho _{0}^{2})^{2}}}\ ds \\
& =-\frac{1}{2}\int_{-\frac{\sqrt{3}}{2}\rho _{0}^{2}}^{r^{2}-\frac{\sqrt{3}%
}{2}\rho _{0}^{2}}\frac{x}{\sqrt{\rho _{0}^{4}-x^{2}}}\ dx,\ \ x=s^{2}-\frac{%
\sqrt{3}}{2}\rho _{0}^{2} \\
& =-\frac{1}{4}\rho _{0}^{2}+\frac{1}{2}\sqrt{\rho _{0}^{4}-\left( r^{2}-%
\frac{\sqrt{3}}{2}\rho _{0}^{2}\right) ^{2}}.
\end{split}%
\end{equation}%
Therefore if $u(0)=\frac{1}{4}\rho _{0}^{2},$ then we have 
\begin{equation}
u(r)=\frac{1}{2}\sqrt{\rho _{0}^{4}-\left( r^{2}-\frac{\sqrt{3}}{2}\rho
_{0}^{2}\right) ^{2}},\ \ 0<r<\sqrt{\frac{2+\sqrt{3}}{2}}\rho _{0}.
\label{gra2}
\end{equation}%
This is part/a graph of a type II shifted Heisenberg sphere. It is not
possible to glue parts of type I and type II shifted Heisenberg spheres to
form a new surface of $C^{2}.$ See Remark \ref{R-3-0} below.
\end{proof}

\begin{remark}
\label{R-3-0} It is easy to see that for a type I shifted Heisenberg sphere $%
u(\sqrt{\frac{2-\sqrt{3}}{2}}\rho _{0})=0$ and after a straightforward
computation for $r_{uu}$ at $u=0$, we have%
\begin{equation}
r_{uu}(0)=\frac{-2\sqrt{2}}{(2-\sqrt{3})^{\frac{1}{2}}\rho _{0}^{3}}
\label{2ndde1}
\end{equation}
From (\ref{gra2}) we see that $u(b)=0$ for some $b>0$ and after a
straightforward computation for $r_{uu}$ at $u=0$, we have 
\begin{equation}
r_{uu}(0)=\frac{-(2+\sqrt{3})}{b^{3}},
\end{equation}%
where the number $b$ is the radius in some sense. If we take the radius $%
b=r_{0}=\sqrt{\frac{2-\sqrt{3}}{2}}\rho _{0}$, then%
\begin{equation}
r_{uu}(0)=\frac{-(2+\sqrt{3})}{b^{3}}=-\left( \frac{2+\sqrt{3}}{2-\sqrt{3}}%
\right) \frac{2\sqrt{2}}{(2-\sqrt{3})^{\frac{1}{2}}\rho _{0}^{3}}.
\label{2ndde2}
\end{equation}%
\ \noindent From (\ref{2ndde1}) and (\ref{2ndde2}), we conclude that the two
graphs (\ref{gra1}) and (\ref{gra2}) can be glued together at $u=0$ to
become a $C^{1}$ smooth curve, but not of $C^{2}$.
\end{remark}

In terms of polar coordinates $(r,\phi )$, we have 
\begin{equation}
\begin{split}
N^{\perp }& =\sin {\theta }\ \partial _{x}-\cos {\theta }\ \partial _{y} \\
& =\sin {\theta }(\cos {\phi }\ \partial _{r}-\frac{\sin {\phi }}{r}\
\partial _{\phi })-\cos {\theta }(\sin {\phi }\ \partial _{r}-\frac{\cos {%
\phi }}{r}\ \partial _{\phi }) \\
& =\sin {(\theta -\phi )}\partial _{r}-\frac{\cos {(\theta -\phi )}}{r}%
\partial _{\phi } \\
N& =\cos {\theta }\ \partial _{x}+\sin {\theta }\ \partial _{y} \\
& =\cos {\theta }(\cos {\phi }\ \partial _{r}-\frac{\sin {\phi }}{r}\
\partial _{\phi })+\sin {\theta }(\sin {\phi }\ \partial _{r}-\frac{\cos {%
\phi }}{r}\ \partial _{\phi }) \\
& =\cos {(\theta -\phi )}\partial _{r}+\frac{\sin {(\theta -\phi )}}{r}%
\partial _{\phi }.
\end{split}
\label{fifuda}
\end{equation}

From \eqref{fifuda}, we have 
\begin{equation*}
\begin{split}
N^{\perp}&=\sin{(\theta-\phi)}\partial_{r}-\frac{\cos{(\theta-\phi)}}{r}%
\partial_{\phi}, \\
N&=\cos{(\theta-\phi)}\partial_{r}+\frac{\sin{(\theta-\phi)}}{r}%
\partial_{\phi},
\end{split}%
\end{equation*}
where 
\begin{equation*}
\begin{split}
\sin{(\theta-\phi)}&=\sin{\theta}\cos{\phi}-\cos{\theta}\sin{\phi} \\
&=\frac{(u_{y}+x)}{D}\frac{x}{r}-\frac{(u_{x}-y)}{D}\frac{y}{r} \\
&=\frac{(xu_{y}-yu_{x})+r^{2}}{rD} \\
&=\frac{r^{2}+u_{\phi}}{rD},
\end{split}%
\end{equation*}
and 
\begin{equation*}
\begin{split}
\cos{(\theta-\phi)}&=\cos{\theta}\cos{\phi}-\sin{\theta}\sin{\phi} \\
&=\frac{(u_{x}-y)}{D}\frac{x}{r}-\frac{(u_{y}+x)}{D}\frac{y}{r} \\
&=\frac{(xu_{x}+yu_{y})}{rD} \\
&=\frac{ru_{r}}{rD}=\frac{u_{r}}{D}.
\end{split}%
\end{equation*}
Therefore 
\begin{equation}  \label{fifuda1}
\begin{split}
N^{\perp}&=\left(\frac{r^{2}+u_{\phi}}{rD}\right)\partial_{r}-\left(\frac{%
u_{r}}{rD}\right)\partial_{\phi}, \\
N&=\left(\frac{u_{r}}{D}\right)\partial_{r}+\left(\frac{r^{2}+u_{\phi}}{%
r^{2}D}\right)\partial_{\phi}.
\end{split}%
\end{equation}

\begin{example}
Let $u=\pm\frac{\sqrt{3}}{2}r^{2}$. We have 
\begin{equation*}
u_{x}-y=\pm\sqrt{3}x-y,\ \ u_{y}+x=x\pm\sqrt{3}y,
\end{equation*}
thus 
\begin{equation*}
D=\sqrt{(u_{x}-y)^{2}+(u_{y}+x)^{2}}=\sqrt{4(x^{2}+y^{2})}=2r,
\end{equation*}
and hence 
\begin{equation*}
\alpha=-\frac{1}{D}=-\frac{1}{2r},
\end{equation*}
and 
\begin{equation*}
\begin{split}
N^{\perp}&=\left(\frac{r^{2}+u_{\phi}}{rD}\right)\partial_{r}-\left(\frac{%
u_{r}}{rD}\right)\partial_{\phi}, \\
&=\frac{1}{2}\partial_{r}\mp\frac{\sqrt{3}}{2r}\partial_{\phi}.
\end{split}%
\end{equation*}
From \eqref{fifuda}, we have 
\begin{equation*}
s=\sin{(\theta-\phi)}=\frac{1}{2},\ c=\cos{(\theta-\phi)}=\pm\frac{\sqrt{3}}{%
2}.
\end{equation*}
This implies that 
\begin{equation*}
\theta=\left\{%
\begin{array}{ll}
\phi+\frac{\pi}{6}, & \ \ \text{for}\ u=\frac{\sqrt{3}}{2}r^{2} \\ 
\phi+\frac{5\pi}{6}, & \ \ \text{for}\ u=-\frac{\sqrt{3}}{2}r^{2}%
\end{array}%
\right.
\end{equation*}
and hence 
\begin{equation*}
\begin{split}
H&=-N^{\perp}\theta=\pm\frac{\sqrt{3}}{2r}, \\
-N^{\perp}\alpha&=-\frac{1}{4r^{2}},
\end{split}%
\end{equation*}
Thus 
\begin{equation*}
e_{1}(\alpha)+\frac{1}{2}\alpha^{2}+\frac{1}{6}H^{2}=0,\ \ \text{where}\
e_{1}=-N^{\perp}.
\end{equation*}
On the other hand, we compute 
\begin{equation*}
\det{A(r,\phi,U)}=s^{2}(\frac{1}{3}sH+2c\alpha)(\frac{2}{3}sH+2c\alpha)=%
\frac{5}{48r^2}\neq 0,\ \ \text{for}\ \ r>0.
\end{equation*}
Therefore each curve defined by $r=\mathbf{c}>0$ is a non-characteristic
curve.
\end{example}


\begin{figure}[h]
\includegraphics[width=10.4cm]{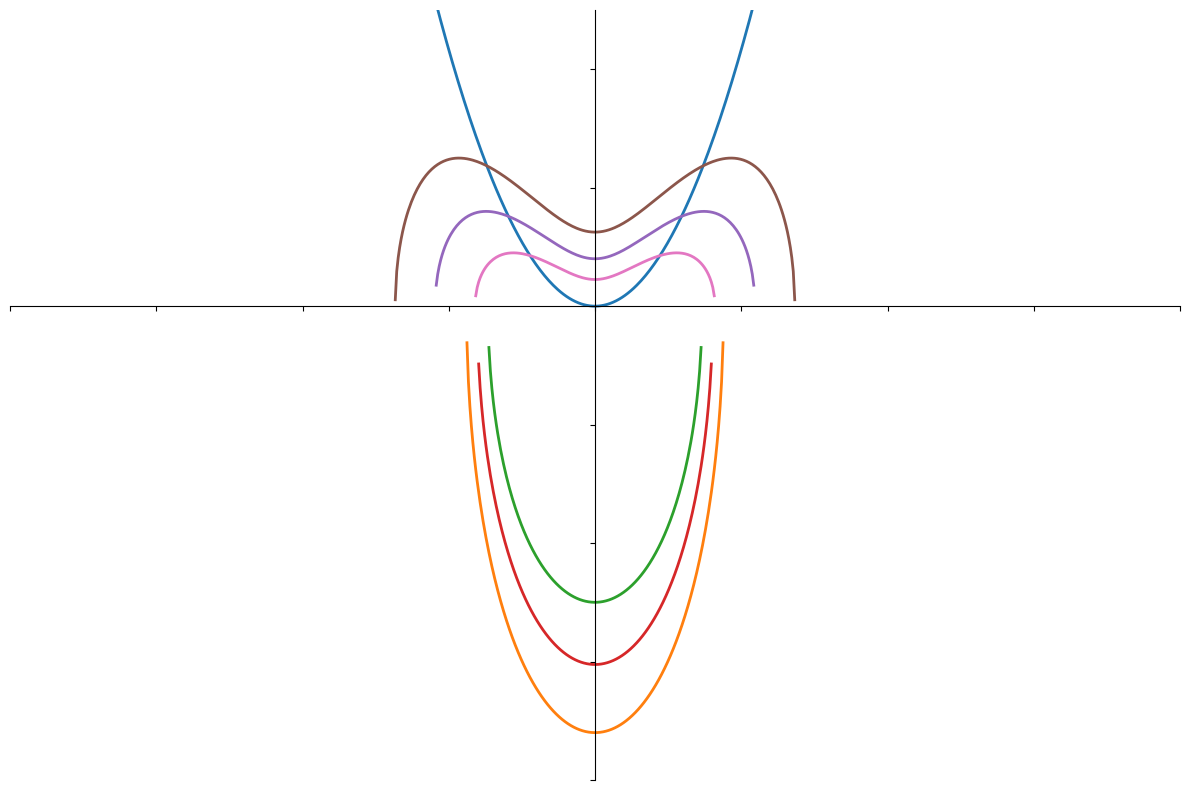}
\caption{(\protect\ref{3-20})}
\end{figure}

\begin{example}
For a type I shifted Heisenberg sphere and a type II shifted Heisenberg
sphere, we let 
\begin{equation}
u(r)=\mp\frac{1}{2}\sqrt{\rho_{0}^{4}-\left(r^{2}\pm\frac{\sqrt{3}}{2}%
\rho_{0}^{2}\right)^{2}},\ \ 0<r<\sqrt{\frac{2\mp\sqrt{3}}{2}}\rho_{0},
\label{3-20}
\end{equation}
which, respectively, describes the lower semi-sphere of type I and the upper
semi-sphere of type II. We have 
\begin{equation*}
u_{r}=\frac{\pm\left(r^{2}\pm\frac{\sqrt{3}}{2}\rho_{0}^{2}\right)}{\sqrt{%
\rho_{0}^{4}-\left(r^{2}\pm\frac{\sqrt{3}}{2}\rho_{0}^{2}\right)^{2}}},
\end{equation*}
thus 
\begin{equation*}
D=\sqrt{r^{2}+u_{r}^{2}}=\frac{r\rho_{0}^{2}}{\sqrt{\rho_{0}^{4}-\left(r^{2}%
\pm\frac{\sqrt{3}}{2}\rho_{0}^{2}\right)^{2}}},
\end{equation*}
and hence 
\begin{equation*}
\alpha=-\frac{1}{D},
\end{equation*}%
and 
\begin{equation*}
\begin{split}
N^{\perp}&=\left(\frac{r^{2}+u_{\phi}}{rD}\right)\partial_{r}-\left(\frac{%
u_{r}}{rD}\right)\partial_{\phi}, \\
&=\left(\frac{r}{D}\right)\partial_{r}-\left(\frac{u_{r}}{rD}%
\right)\partial_{\phi},\ \ \ u_{\phi}=0 \\
&=\frac{\sqrt{\rho_{0}^{4}-\left(r^{2}\pm\frac{\sqrt{3}}{2}%
\rho_{0}^{2}\right)^{2}}}{\rho_{0}^{2}}\partial_{r}\mp\frac{\left(r^{2}\pm%
\frac{\sqrt{3}}{2}\rho_{0}^{2}\right)}{r\rho_{0}^{2}}\partial_{\phi}.
\end{split}%
\end{equation*}
From \eqref{fifuda}, we have 
\begin{equation*}
s=\sin{(\theta-\phi)}=\frac{\sqrt{\rho_{0}^{4}-\left(r^{2}\pm\frac{\sqrt{3}}{%
2}\rho_{0}^{2}\right)^{2}}}{\rho_{0}^{2}},
\end{equation*}
and 
\begin{equation*}
c=\cos{(\theta-\phi)}=\pm\frac{\left(r^{2}\pm\frac{\sqrt{3}}{2}%
\rho_{0}^{2}\right)}{\rho_{0}^{2}}.
\end{equation*}
Therefore 
\begin{equation*}
\theta-\phi=\sin^{-1}\left(\frac{\sqrt{\rho_{0}^{4}-\left(r^{2}\pm\frac{%
\sqrt{3}}{2}\rho_{0}^{2}\right)^{2}}}{\rho_{0}^{2}}\right).
\end{equation*}
We compute the $p$-mean curvature 
\begin{equation*}
\begin{split}
H&=-N^{\perp}\theta \\
&=\pm\frac{\left(r^{2}\pm\frac{\sqrt{3}}{2}\rho_{0}^{2}\right)}{r\rho_{0}^{2}%
}-\frac{\sqrt{\rho_{0}^{4}-\left(r^{2}\pm\frac{\sqrt{3}}{2}%
\rho_{0}^{2}\right)^{2}}}{\rho_{0}^{2}}\partial_{r}\left(\sin^{-1}\frac{%
\sqrt{\rho_{0}^{4}-\left(r^{2}\pm\frac{\sqrt{3}}{2}\rho_{0}^{2}\right)^{2}}}{%
\rho_{0}^{2}}\right),
\end{split}%
\end{equation*}
where, by chain rules, 
\begin{equation*}
\begin{split}
\partial_{r}\left(\sin^{-1}\frac{\sqrt{\rho_{0}^{4}-\left(r^{2}\pm\frac{%
\sqrt{3}}{2}\rho_{0}^{2}\right)^{2}}}{\rho_{0}^{2}}\right)&=\frac{1}{\cos{%
(\theta-\phi)}}\partial_{r}\left(\frac{\sqrt{\rho_{0}^{4}-\left(r^{2}\pm%
\frac{\sqrt{3}}{2}\rho_{0}^{2}\right)^{2}}}{\rho_{0}^{2}}\right) \\
&=\frac{1}{\cos{(\theta-\phi)}}\left(\frac{-2r\left(r^{2}\pm\frac{\sqrt{3}}{2%
}\rho_{0}^{2}\right)}{\rho_{0}^{2}\sqrt{\rho_{0}^{4}-\left(r^{2}\pm\frac{%
\sqrt{3}}{2}\rho_{0}^{2}\right)^{2}}}\right).
\end{split}%
\end{equation*}
Therefore 
\begin{equation*}
\begin{split}
H&=-N^{\perp}\theta \\
&=\pm\frac{\left(r^{2}\pm\frac{\sqrt{3}}{2}\rho_{0}^{2}\right)}{r\rho_{0}^{2}%
}-\frac{1}{\cos{(\theta-\phi)}}\left(\frac{-2r\left(r^{2}\pm\frac{\sqrt{3}}{2%
}\rho_{0}^{2}\right)}{\rho_{0}^{4}}\right) \\
&=\pm\frac{\left(r^{2}\pm\frac{\sqrt{3}}{2}\rho_{0}^{2}\right)}{r\rho_{0}^{2}%
}\pm\frac{2r}{\rho_{0}^{2}} \\
&=\pm\frac{\left(3r^{2}\pm\frac{\sqrt{3}}{2}\rho_{0}^{2}\right)}{%
r\rho_{0}^{2}}.
\end{split}%
\end{equation*}
Also, after a straightforwards computation, we have 
\begin{equation*}
-N^{\perp}\alpha=\frac{4}{\rho_{0}^{4}}\left(-\frac{u^{2}}{r^{2}}-\frac{%
r^{2}\pm\frac{\sqrt{3}}{2}\rho_{0}^{2}}{2}\right).
\end{equation*}
Thus 
\begin{equation*}
e_{1}(\alpha)+\frac{1}{2}\alpha^{2}+\frac{1}{6}H^{2}=0,\ \ \text{where}\
e_{1}=-N^{\perp}.
\end{equation*}
\end{example}

\begin{remark}
Both the geometric invariants $\alpha$ and $H$ we compute in examples are
with respect to the horizontal normal 
\begin{equation*}
e_{2}=\frac{\nabla_{b}\psi}{\|\nabla_{b}\psi\|},
\end{equation*}
where $\psi$ is the defining function defined by 
\begin{equation*}
\psi=t-u.
\end{equation*}
With respect to such a horizontal normal $e_{2}$, we have

\begin{enumerate}
\item The function $\alpha$ is always negative, i.e., $\alpha=-\frac{1}{D}%
\leq 0$.

\item $e_{1}=-N^{\perp }$, when projected onto the $xy$-plane.
\end{enumerate}

If we choose the horizontal normal $\tilde{e}_{2}$ such that 
\begin{equation*}
\tilde{e}_{2}=\frac{\nabla_{b}\tilde{\psi}}{\|\nabla_{b}\tilde{\psi}\|}
\end{equation*}
with the defining function 
\begin{equation*}
\tilde{\psi}=u^{2}-t^{2}.
\end{equation*}
Since 
\begin{equation*}
\begin{split}
\nabla_{b}\tilde{\psi}&=(u+t)\nabla_{b}(u-t)+(u-t)\nabla_{b}(u+t) \\
&=2u\nabla_{b}(u-t),\ \ \ \text{on}\ t-u=0,
\end{split}%
\end{equation*}
we see that on the surface $\Sigma$ defined by the graph of $t=u$, we have 
\begin{equation*}
\tilde{e}_{2}=\left\{%
\begin{array}{ll}
-e_{2}, & \ \ \text{when}\ u>0 \\ 
e_{2}, & \ \ \text{when}\ u<0%
\end{array}%
\right.
\end{equation*}
Therefore the invariants $\alpha$ and $H$ with respect to $\tilde{e}_{2}$
will have the same sign as those with respect to $e_{2}$ when $u<0$; and
will differ by a sign when $u>0$. Notice that, for a Pansu sphere, the
horizontal normal $\tilde{e}_{2}$ is globally defined on the whole regular
part and the corresponding $p$-mean curvature $H$ is a positive constant.
\end{remark}

\begin{figure}[h]
\includegraphics[width=10.4cm]{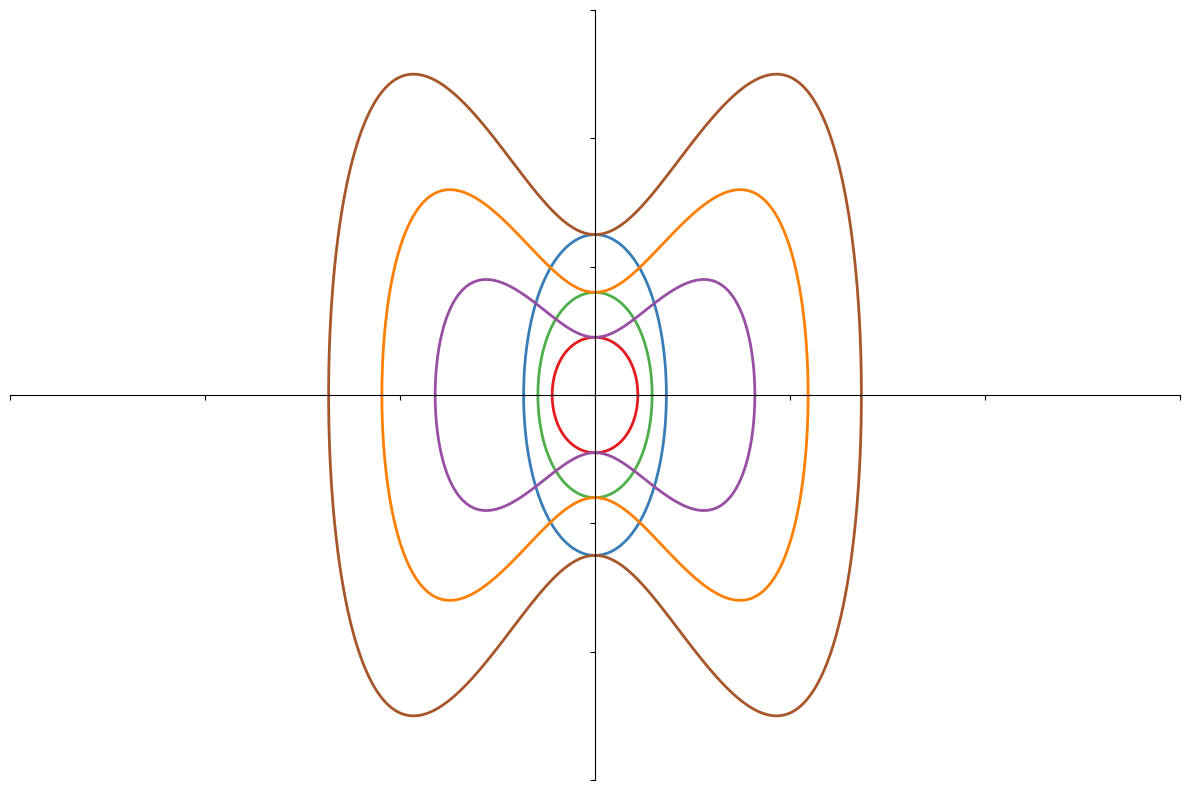}
\caption{closed pictures}
\end{figure}

\begin{remark}
\label{R-3-1} The solutions $u(r)=\pm \frac{\sqrt{3}}{2}r^{2}$ in (1)/(2) of
Theorem \ref{T-class} can be obtained as a limit of dilations applied to
type I/II shifted Heisenberg spheres (see Figure 3 for closed pictures)
after translating their \textquotedblleft south pole" to the origin. Let $%
\tilde{t}=t+\frac{1}{4}\rho _{0}^{2}.$ The surfaces defined by $(r^{2}\pm 
\frac{\sqrt{3}}{2}\rho _{0}^{2})^{2}+4t^{2}=\rho _{0}^{4},$ $\rho _{0}>0,$
can be translated to the form $(r^{2}\pm \frac{\sqrt{3}}{2}\rho
_{0}^{2})^{2}+4(\tilde{t}-\frac{1}{4}\rho _{0}^{2})^{2}=\rho _{0}^{4}$
passing through $(r=0,\tilde{t}=0).$ Now applying the dilations $(r,\tilde{t}%
)$ $\rightarrow $ $(\lambda r,\lambda ^{2}\tilde{t}),$ $\lambda >0,$ and
letting $\lambda \rightarrow 0$ after dividing by $\lambda ^{2}$ yield%
\begin{equation*}
\pm \sqrt{3}\rho _{0}^{2}r^{2}-2\rho _{0}^{2}\tilde{t}=0
\end{equation*}%
which are solutions (1)/(2) of Theorem \ref{T-class}.
\end{remark}

\section{Proof of Theorem \protect\ref{T-2nd-var}: second variation of $%
E_{1} \label{Sec4}$ for the Clifford torus}

\subsection{First variation of $E_{1}:$review}

We first review some material in \cite{CYZ18}, including the first variation
of $E_{1}$ and lemmas which will be used to compute the second variation of $%
E_{1}$ for the Clifford torus in the standard CR 3-sphere, which is a
critical point of $E_1$.

Let $\Sigma _{t}=F_{t}(\Sigma )$ be a family of surfaces in a 3-dimensional
pseudohermitian manifold ($M,J,\Theta $) with $\Sigma _{0}=\Sigma ,$ such
that 
\begin{equation}
\frac{d}{dt}F_{t}=X=fe_{2}+gT.  \label{fg}
\end{equation}%
\noindent Here we let $e_{1}$ be the unit vector in $T\Sigma _{t}\cap \xi $
and $e_{2}=Je_{1}$. We assume $f$ and $g$ are supported in a domain of $%
\Sigma $ away from the singular set of $\Sigma $. We denote $h:=f-\alpha g$
and $V:=T+\alpha e_2$. As satisfied by the standard CR 3-sphere, we assume $%
(M,J,\Theta )$ has vanishing torsion and constant Webster scalar curvature
in this section.

\begin{theorem}
\label{firstvariation1} \label{CYZ18} (\cite[Theorem 3]{CYZ18}) Assume $%
(M,J,\Theta )$ has vanishing torsion and constant Webster scalar curvature.
Let $F_{t}(\Sigma )$ be given by (\ref{fg}). We have 
\begin{equation}
\frac{d}{dt}E_1[\Sigma_t]=\frac{d}{dt}\int_{\Sigma_t}dA_{1} =\int_{\Sigma_t }%
\mathcal{E}_{1}h\Theta \wedge e^{1},  \label{firstvariationofE1}
\end{equation}%
where 
\begin{equation*}
dA_1=|H_{cr}|^{3/2}\Theta\wedge e^1, \quad H_{cr}:=e_{1}(\alpha )+\frac{1}{2}%
\alpha ^{2}+\frac{1}{6}H^{2}+\frac{1}{4}W,
\end{equation*}
\begin{eqnarray}
\mathcal{E}_1&=&\frac{1}{2}e_{1}(|H_{cr}|^{1/2}\mathfrak{f})+\frac{3}{2}%
|H_{cr}|^{1/2}\alpha \mathfrak{f}  \label{Epsilon1-0} \\
&&+\frac{1}{2}sign(H_{cr})|H_{cr}|^{1/2}\{9V(\alpha)+6HH_{cr}-\frac{1}{3}%
H^3\},  \notag
\end{eqnarray}
and $\mathfrak{f}$ is the function on $\Sigma_t$ given by (see \cite[(3.11)]%
{CYZ18}) 
\begin{eqnarray*}
\mathfrak{f} &\mathfrak{=}&|H_{cr}|^{-1}\{e_{1}(H)(e_{1}(\alpha )+\frac{1}{2}%
\alpha ^{2}+\frac{1}{3}H^{2}+\frac{1}{4}W) \\
&&+HV(H)+\frac{3}{2}V(e_{1}(\alpha )+\frac{1}{2}\alpha ^{2}) \\
&&-\frac{7}{2}\alpha He_{1}(\alpha )-\frac{5}{2}\alpha ^{3}H-\frac{2}{3}%
\alpha H^{3}-\frac{5}{4}\alpha HW\}.
\end{eqnarray*}
\end{theorem}

We rewrite $|H_{cr}|\mathfrak{f}$ and $\mathcal{E}_1$ in the following form:

\begin{proposition}
\label{firstvariation2} We have 
\begin{equation*}
|H_{cr}|\mathfrak{f}\mathfrak{=}e_{1}(H)H_{cr}+\frac{3}{2}V(H_{cr})+\frac{1}{%
2}He_{1}(H_{cr})-\alpha HH_{cr},
\end{equation*}
\end{proposition}

and 
\begin{eqnarray}
\mathcal{E}_1&=&|H_{cr}|^{-\frac{1}{2}}\{-\frac{1}{4}sign(H_{cr})\mathfrak{f}%
e_1(H_{cr})+\frac{1}{2}e_1(|H_{cr}|\mathfrak{f})  \label{Epsilon1} \\
&&+\frac{3}{2}|H_{cr}|\mathfrak{f}\alpha+H_{cr}[\frac{9}{2}%
V(\alpha)+3HH_{cr}-\frac{1}{6}H^3]\}.  \notag
\end{eqnarray}
Proof. We compute 
\begin{eqnarray*}
|H_{cr}|\mathfrak{f} &\mathfrak{=}&e_{1}(H)(H_{cr}+\frac{1}{6}H^{2}) \\
&&+HV(H)+\frac{3}{2}V(H_{cr}-\frac{1}{6}H^{2}) \\
&&-\frac{7}{2}\alpha He_{1}(\alpha )-\frac{5}{2}\alpha ^{3}H-\frac{2}{3}%
\alpha H^{3}-\frac{5}{4}\alpha HW \\
&=&e_{1}(H)H_{cr}+\frac{3}{2}V(H_{cr}) \\
&&+\frac{1}{2}He_{1}(H_{cr}-e_{1}(\alpha )-\frac{1}{2}\alpha ^{2})+\frac{1}{2%
}HV(H) \\
&&-\frac{7}{2}\alpha He_{1}(\alpha )-\frac{5}{2}\alpha ^{3}H-\frac{2}{3}%
\alpha H^{3}-\frac{5}{4}\alpha HW.
\end{eqnarray*}
Then, using the Codazzi-like equation (\cite[(3.24)]{CYZ18}) 
\begin{equation*}
e_1e_1(\alpha)=-6\alpha e_1(\alpha)+V(H)-\alpha H^2-4\alpha^2-2W\alpha,
\end{equation*}
we find 
\begin{eqnarray*}
|H_{cr}|\mathfrak{f} &=&e_{1}(H)H_{cr}+\frac{3}{2}V(H_{cr})+\frac{1}{2}%
He_{1}(H_{cr}) \\
&&-\frac{1}{2}H[-5\alpha e_{1}(\alpha )-\alpha H^{2}-4\alpha ^{3}-2W\alpha ]
\\
&&-\frac{7}{2}\alpha He_{1}(\alpha )-\frac{5}{2}\alpha ^{3}H-\frac{2}{3}%
\alpha H^{3}-\frac{5}{4}\alpha HW \\
&=&e_{1}(H)H_{cr}+\frac{3}{2}V(H_{cr})+\frac{1}{2}He_{1}(H_{cr}) \\
&&-H\alpha \lbrack e_{1}(\alpha )+\frac{1}{2}\alpha ^{2}+\frac{1}{6}H^{2}+%
\frac{1}{4}W] \\
&=&e_{1}(H)H_{cr}+\frac{3}{2}V(H_{cr})+\frac{1}{2}He_{1}(H_{cr})-\alpha
HH_{cr}.
\end{eqnarray*}

We rewrite (\ref{Epsilon1-0}) as 
\begin{eqnarray*}
\mathcal{E}_1&=&\frac{1}{2}e_{1}(|H_{cr}|^{1/2}\mathfrak{f)}+\frac{3}{2}%
|H_{cr}|^{1/2}\alpha \mathfrak{f} \\
&&+\frac{1}{2}sign(H_{cr})|H_{cr}|^{1/2}\{9V(\alpha )+6HH_{cr}-\frac{1}{3}%
H^{3}\} \\
&=&|H_{cr}|^{-\frac{1}{2}}\{ \frac{3}{2}|H_{cr}|\mathfrak{f}\alpha+\frac{1}{2%
}H_{cr}[9V(\alpha )+6HH_{cr}-\frac{1}{3}H^{3}] \\
&&+\frac{1}{2}e_1(|H_{cr}|\mathfrak{f})-\frac{1}{4}sign(H_{cr})\mathfrak{f}
e_1(H_{cr}) \}.
\end{eqnarray*}
$\hfill \Box$

\subsection{Second variation of $E_{1}$}

We assume that $\Sigma \hookrightarrow (M,J,\Theta )$ satisfies $\mathcal{E}%
_{1}(\Sigma )=0$, i.e. a critical point of $E_{1}$, and $H_{cr}\neq 0$. Let $%
\Sigma _{t}=F_{t}(\Sigma )$ be a family of surfaces in ($M,J,\Theta $) with $%
\Sigma _{0}=\Sigma ,$ such that (recall (\ref{fg}))%
\begin{equation*}
\frac{d}{dt}F_{t}=X=fe_{2}+gT.
\end{equation*}
Here we let $e_{1}$ be the unit vector in $T\Sigma _{t}\cap \xi $ and $%
e_{2}=Je_{1}$. We denote $h:=f-\alpha g$ and $V:=T+\alpha e_{2}$. Then by (%
\ref{firstvariationofE1}) and (\ref{Epsilon1}), the second variation of $%
E_{1}$ for $\Sigma $ is given by 
\begin{eqnarray}
\frac{d^{2}}{dt^{2}}|_{t=0}E_{1}(\Sigma _{t}) &=&\int_{\Sigma }|H_{cr}|^{-%
\frac{1}{2}}\frac{d}{dt}\{-\frac{1}{4}|H_{cr}|\mathfrak{f}\frac{1}{H_{cr}}%
e_{1}(H_{cr})+\frac{1}{2}e_{1}(|H_{cr}|\mathfrak{f})
\label{secondvariation1} \\
&&+\frac{3}{2}|H_{cr}|\mathfrak{f}\alpha +H_{cr}[\frac{9}{2}V(\alpha
)+3HH_{cr}-\frac{1}{6}H^{3}]\}h\Theta \wedge e^{1}.\quad   \notag
\end{eqnarray}%
We need some lemmas to compute the above second variation of $E_{1}$ for the
Clifford torus, some were proved in \cite{CYZ18}.

Assume $\{e_1,e_2=Je_1,T\}$ is an orthonormal frame of $(M,J,\Theta)$ with
respect to the adapted metric $g_\Theta=\frac{1}{2}d\Theta(\cdot,J\cdot)+%
\Theta\otimes\Theta$ and $\{e^1,e^2,\Theta\}$ is its dual coframe. There is
a real 1-form $\omega$ such that the following holds (see, for instance,
Appendix of \cite{CYZ18}), under the assumption that $(M,J,\Theta)$ has
vanishing torsion, 
\begin{eqnarray*}
de^{1} &=&-e^{2}\wedge \omega ,\text{ }de^{2}=e^{1}\wedge \omega , \\
\lbrack e_{2},e_{1}] &=&\omega (e_{2})e_{2}+\omega (e_{1})e_{1}+2T, \\
\lbrack e_{1},T] &=&-\omega (T)e_{2},\text{ }[e_{2},T]=\omega (T)e_{1}.
\end{eqnarray*}

\begin{lemma}
(\cite[Lemma 4]{CYZ18}) \label{Lemma4} \label{L-4-1} Let $h:=f-\alpha g$ and 
$V:=$ $T+\alpha e_{2}.$ Under the torsion free condition, we have 
\begin{equation*}
\omega(e_1)=H, \quad \omega (e_{2})=h^{-1}e_{1}(h)+2\alpha, \quad \omega
(T)=e_{1}(\alpha )-\alpha h^{-1}e_{1}(h),  \label{omegaT}
\end{equation*}%
\begin{equation*}
e_{2}(\alpha )=h^{-1}V(h),  \label{e2alpha}
\end{equation*}%
\begin{equation*}
\lbrack e_{1},V]=-\alpha He_{1}-2\alpha V,
\end{equation*}
\end{lemma}

\begin{lemma}
(\cite[Lemma 6]{CYZ18}) \label{Lemma6} \label{L-4-4} Suppose the torsion
vanishes, i.e., $A_{\bar{1}}^{1}$ $=$ $0$. Then we have%
\begin{equation*}
\frac{dH}{dt}=e_{1}e_{1}(h)+2\alpha e_{1}(h)+4h(e_{1}(\alpha )+\alpha ^{2}+%
\frac{1}{4}H^{2}+\frac{1}{2}W)+gV(H),  \label{dHdt}
\end{equation*}%
\begin{equation*}
\frac{d\alpha }{dt}=V(h)+gV(\alpha ),
\end{equation*}%
\begin{equation*}
\frac{d}{dt}e_{1}(\alpha )=e_{1}V(h)+ge_{1}V(\alpha )+2fV(\alpha
)+fHe_{1}(\alpha ).
\end{equation*}
\end{lemma}

In addition, we compute the following

\begin{lemma}
\label{L-4-5} We have%
\begin{equation*}
\frac{d}{dt}V(\alpha )=V(\frac{d\alpha }{dt})+he_{1}(\alpha )e_{1}(\alpha
)-\alpha e_{1}(h)e_{1}(\alpha )-V(g)V(\alpha ),
\end{equation*}%
and 
\begin{equation*}
\frac{d}{dt}[e_{1}(|H_{cr}|\mathfrak{f})]=e_{1}[\frac{d}{dt}(|H_{cr}|%
\mathfrak{f})]+[fHe_{1}+2fV-e_{1}(g)V](|H_{cr}|\mathfrak{f}),
\end{equation*}
where 
\begin{eqnarray*}
\frac{d}{dt}(|H_{cr}|\mathfrak{f}) &=&H_{cr}\frac{d}{dt}[e_{1}(H)]+e_{1}(H)%
\frac{d}{dt}H_{cr}+\frac{3}{2}\frac{d}{dt}V(H_{cr}) \\
&&+\frac{1}{2}e_{1}(H_{cr})\frac{dH}{dt}+\frac{1}{2}H\frac{d}{dt}%
[e_{1}(H_{cr})] \\
&&-\frac{d\alpha }{dt}HH_{cr}-\alpha \frac{dH}{dt}H_{cr}-\alpha H\frac{%
dH_{cr}}{dt},
\end{eqnarray*}
here 
\begin{eqnarray*}
\frac{d}{dt}e_{1}(H) &=&e_{1}[e_{1}e_{1}(h)+2\alpha e_{1}(h)+4he_{1}(\alpha
)+H^{2}h+4\alpha ^{2}h+2Wh] \\
&&+ge_{1}V(H)+fHe_{1}(H)+2fV(H),
\end{eqnarray*}
\begin{equation*}
\frac{d}{dt}V(H_{cr})=V(\frac{dH_{cr}}{dt})+he_{1}(\alpha
)e_{1}(H_{cr})-\alpha e_{1}(h)e_{1}(H_{cr})-V(g)V(H_{cr}),
\end{equation*}
and 
\begin{eqnarray*}
\frac{d}{dt}H_{cr} &=&e_{1}V(h)+\alpha V(h)+\frac{1}{3}H[e_{1}e_{1}(h)+2%
\alpha e_{1}(h)] \\
&&+\frac{4}{3}H(e_{1}(\alpha )+\alpha ^{2}+\frac{1}{4}H^{2}+\frac{1}{2}W)h \\
&&+(2V(\alpha )+He_{1}(\alpha ))f+(e_{1}V(\alpha )+\alpha V(\alpha )+\frac{1%
}{3}HV(H))g.
\end{eqnarray*}
\end{lemma}

Proof. It follows from Lemma \ref{Lemma4} that 
\begin{eqnarray*}
[fe_2+gT,V]&=&f[e_2,V]+g[T,V]-V(f)e_2-V(g)T \\
&=&f[e_2(\alpha)e_2 +e_1(\alpha) e_1 -\alpha h^{-1}e_1(h)e_1] \\
&&+g[T(\alpha)e_2-\alpha e_1(\alpha)e_1+\alpha^2h^{-1}e_1(h)e_1 ] \\
&&-V(f)e_2-V(g)T \\
&=&he_1(\alpha)e_1-\alpha e_1(h)e_1-V(g)V.
\end{eqnarray*}
Then 
\begin{equation*}
\frac{d}{dt}V(\alpha )=V(\frac{d\alpha }{dt})+he_{1}(\alpha )e_{1}(\alpha
)-\alpha e_{1}(h)e_{1}(\alpha )-V(g)V(\alpha ).
\end{equation*}
In a similar manner, 
\begin{eqnarray*}
\frac{d}{dt}V(H_{cr}) &=&V(\frac{d}{dt}H_{cr})+[fe_{2}+gT,V](H_{cr}) \\
&=&V(\frac{dH_{cr}}{dt})+he_{1}(\alpha)e_{1}(H_{cr})-\alpha
e_{1}(h)e_{1}(H_{cr})-V(g)V(H_{cr}).
\end{eqnarray*}

It follows from Lemma \ref{Lemma4} that 
\begin{eqnarray*}
\lbrack fe_{2}+gT,e_{1}]&=&f[e_2,e_1]-e_1(f)e_2+g[T,e_1]-e_1(g)T \\
&=&f[\omega(e_2)e_2+\omega(e_1)e_1+2T]-e_1(f)e_2+g\omega(T)e_2-e_1(g)T  \\
&=&fHe_{1}+2fV-e_{1}(g)V,
\end{eqnarray*}%
then by Lemma \ref{Lemma6} 
\begin{eqnarray*}
\frac{d}{dt}e_{1}(H) &=&e_{1}(\frac{dH}{dt})+[fe_{2}+gT,e_{1}](H) \\
&=&e_{1}[e_{1}e_{1}(h)+2\alpha e_{1}(h)+4h(e_{1}(\alpha )+\alpha ^{2}+\frac{1%
}{4}H^{2}+\frac{1}{2}W)+gV(H)] \\
&&+fHe_{1}(H)+2fV(H)-e_{1}(g)V(H) \\
&=&e_{1}[e_{1}e_{1}(h)+2\alpha e_{1}(h)+4he_{1}(\alpha )+H^{2}h+4\alpha
^{2}h+2Wh] \\
&&+ge_{1}V(H)+fHe_{1}(H)+2fV(H).
\end{eqnarray*}
In a similar manner, we have 
\begin{equation*}
\frac{d}{dt}[e_{1}(|H_{cr}|\mathfrak{f})]=e_{1}[\frac{d}{dt}(|H_{cr}|%
\mathfrak{f})]+[fHe_{1}+2fV-e_{1}(g)V](|H_{cr}|\mathfrak{f}).
\end{equation*}
Recall that 
\begin{equation*}
|H_{cr}|\mathfrak{f}\mathfrak{=}e_{1}(H)H_{cr}+\frac{3}{2}V(H_{cr})+\frac{1}{%
2}He_{1}(H_{cr})-\alpha HH_{cr}.
\end{equation*}%
So we have 
\begin{eqnarray*}
\frac{d}{dt}(|H_{cr}|\mathfrak{f}) &=&H_{cr}\frac{d}{dt}[e_{1}(H)]+e_{1}(H)%
\frac{d}{dt}H_{cr}+\frac{3}{2}\frac{d}{dt}V(H_{cr}) \\
&&+\frac{1}{2}e_{1}(H_{cr})\frac{dH}{dt}+\frac{1}{2}H\frac{d}{dt}%
[e_{1}(H_{cr})] \\
&&-\frac{d\alpha }{dt}HH_{cr}-\alpha \frac{dH}{dt}H_{cr}-\alpha H\frac{%
dH_{cr}}{dt}.
\end{eqnarray*}

Recall that 
\begin{equation*}
H_{cr}=e_{1}(\alpha )+\frac{1}{2}\alpha ^{2}+\frac{1}{6}H^{2}+\frac{1}{4}W.
\end{equation*}
Then by Lemma \ref{Lemma6}, we have 
\begin{eqnarray*}
\frac{d}{dt}H_{cr} &=&e_{1}V(h)+\alpha V(h)+\frac{1}{3}H[e_{1}e_{1}(h)+2%
\alpha e_{1}(h)] \\
&&+\frac{4}{3}H(e_{1}(\alpha )+\alpha ^{2}+\frac{1}{4}H^{2}+\frac{1}{2}W)h \\
&&+(2V(\alpha )+He_{1}(\alpha ))f+(e_{1}V(\alpha )+\alpha V(\alpha )+\frac{1%
}{3}HV(H))g.
\end{eqnarray*}
$\hfill \Box$

\subsection{Second variation of $E_{1}$ for the Clifford torus: proof of
Theorem \protect\ref{T-2nd-var}\label{Sub4-3}}

The Clifford torus of the standard CR 3-sphere $S^{3}$ have been studied in 
\cite[Example 4 in Subsection 4.1]{CYZ18}. The contact form $\hat{\Theta}$
on the standard CR 3-sphere $S^{3}$ reads as $\hat{\Theta}%
=\sum_{i=1}^{2}(x^{i}dy^{i}-y^{i}dx^{i}),$ hence $d\hat{\Theta}%
=2\sum_{i=1}^{2}dx^{i}\wedge dy^{i}$. So the Reeb vector field $%
T=\sum_{i=1}^2 (x^{i}\frac{\partial }{\partial y^{i}}-y^{i}\frac{\partial }{%
\partial x^{i}})$. The standard CR 3-sphere $S^{3}$ has vanishing torsion
and Webster scalar curvature $W=2.$ A (nonsingular) surface $\Sigma \subset
S^{3}$ has $\alpha =0$ if and only if $T\in T\Sigma $. The following
examples have the property that $\alpha =0$: 
\begin{equation*}
\Sigma _{\lbrack \rho _{1}]}:=\{|z^{1}|^{2}=(x^{1})^{2}+(y^{1})^{2}=\rho
_{1}^{2},\quad |z^{2}|^{2}=\rho _{2}^{2}=1-\rho _{1}^{2}\}.
\end{equation*}

We introduce the coordinates $(\varphi_1,\varphi_2)$ for $\Sigma _{\lbrack
\rho _{1}]}$ such that $\Sigma _{\lbrack \rho
_{1}]}=\{(\rho_1e^{i\varphi_1},\rho_2e^{i\varphi_2})\}$. For $\Sigma
_{\lbrack \rho _{1}]}$, we have 
\begin{equation*}
e_{1}=-\frac{\rho _{2}}{\rho _{1}}\frac{\partial }{\partial \varphi _{1}}+%
\frac{\rho _{1}}{\rho _{2}} \frac{\partial }{\partial \varphi _{2}},
\end{equation*}
where $\frac{\partial }{\partial \varphi _{i}}=x^{i}\frac{\partial }{%
\partial y^{i}}-y^{i}\frac{\partial }{\partial x^{i}},$ 
Note also that $V=T=\partial _{\varphi _{1}}+\partial _{\varphi _{2}}$. The
p-mean curvature of $\Sigma _{\lbrack \rho _{1}]}$ is given by 
\begin{equation*}
H=\frac{\rho _{1}^{2}-\rho _{2}^{2}}{\rho _{1}\rho _{2}}=\frac{\rho _{1}}{%
\rho _{2}}-\frac{\rho _{2}}{\rho _{1}}.
\end{equation*}%
When $\rho _{1}=\rho _{2}=\frac{1}{\sqrt{2}}$, it is the Clifford torus $%
\Sigma _{1/\sqrt{2}}$ with 
\begin{equation*}
\alpha =H=0.
\end{equation*}
Besides, for the Clifford torus 
\begin{equation*}
e_{1}=-\frac{\partial }{\partial \varphi _{1}}+\frac{\partial }{\partial
\varphi _{2}}, \quad V=T=\partial _{\varphi_{1}}+\partial _{\varphi _{2}},
\quad [e_{1},V]=0.
\end{equation*}
For the Clifford torus, we compute 
\begin{equation}  \label{HcrClifford}
H_{cr}=e_{1}(\alpha )+\frac{1}{2}\alpha ^{2}+\frac{1}{6}H^{2}+\frac{1}{4}W=%
\frac{1}{2},
\end{equation}
\begin{equation}  \label{HcrfClifford}
|H_{cr}|\mathfrak{f}\mathfrak{=}e_{1}(H)H_{cr}+\frac{3}{2}V(H_{cr})+\frac{1}{%
2}He_{1}(H_{cr})-\alpha HH_{cr}=0.
\end{equation}

We now compute the second variation of $E_1$ for the Clifford torus $\Sigma
_{\lbrack 1/\sqrt{2}]}$, i.e. (\ref{SVF}) of Theorem \ref{T-2nd-var}.

\begin{theorem}
Let $\Sigma $ be the Clifford torus $\Sigma _{\lbrack 1/\sqrt{2}]}$ in the
CR 3-sphere $S^{3}$. Then it holds that%
\begin{equation*}
\frac{d^{2}}{dt^{2}}|_{t=0}E_{1}(\Sigma _{t})=\frac{\sqrt{2}}{4}\int_{\Sigma
_{\lbrack 1/\sqrt{2}%
]}}[e_{1}e_{1}(f)^{2}+3e_{1}T(f)^{2}-7e_{1}(f)^{2}+9fTT(f)+12f^{2}]\Theta
\wedge e^{1}.
\end{equation*}
\end{theorem}

Proof. By (\ref{secondvariation1}) (\ref{HcrClifford}) and (\ref%
{HcrfClifford}), we have the second variation of $E_1$ at $\Sigma
_{0}=\Sigma _{\lbrack 1/\sqrt{2}]}:$ the Clifford torus 
\begin{eqnarray}
&&\frac{d^{2}}{dt^{2}}|_{t=0}E_{1}(\Sigma _{t}) =\int_{\Sigma _{\lbrack 1/%
\sqrt{2}]}}|H_{cr}|^{-\frac{1}{2}}\frac{d}{dt}\{-\frac{1}{4}|H_{cr}|%
\mathfrak{f}\frac{1}{H_{cr}}e_{1}(H_{cr})+\frac{1}{2}e_{1}(|H_{cr}|\mathfrak{%
f})  \notag \\
&&+\frac{3}{2}|H_{cr}|\mathfrak{f}\alpha +H_{cr}[\frac{9}{2}V(\alpha
)+3HH_{cr}-\frac{1}{6}H^{3}]\}h\Theta \wedge e^{1}  \notag \\
&&=\sqrt{2}\int_{\Sigma _{\lbrack 1/\sqrt{2}]}}[\frac{1}{2}\frac{d}{dt}%
e_{1}(|H_{cr}|\mathfrak{f})+\frac{9}{4}\frac{d}{dt}V(\alpha )+\frac{3}{4}%
\frac{dH}{dt}]h\Theta \wedge e^{1}.  \notag
\end{eqnarray}
It follows from Lemma \ref{Lemma6} that 
\begin{equation*}
\frac{d}{dt}|_{t=0}H=e_1e_1(h)+4h.
\end{equation*}
If follows from Lemma \ref{Lemma6} and Lemma \ref{L-4-5} that 
\begin{equation*}
\frac{d}{dt}|_{t=0}V(\alpha)=V\frac{d\alpha}{dt}=VV(h).
\end{equation*}
Using Lemma \ref{L-4-5}, we compute 
\begin{eqnarray*}
\frac{d}{dt}|_{t=0}[e_{1}(|H_{cr}|\mathfrak{f})] &=&e_{1}[\frac{d}{dt}%
(|H_{cr}|\mathfrak{f})] \\
&=&e_1\{ H_{cr}\frac{d}{dt}[e_{1}(H)]+\frac{3}{2}\frac{d}{dt}V(H_{cr})\} \\
&=&e_1\{\frac{1}{2}[e_1e_1e_1(h)+4e_1(h)]+\frac{3}{2}V[\frac{d}{dt}H_{cr}]\}
\\
&=&\frac{1}{2}e_{1}e_{1}e_{1}e_{1}(h)+2e_{1}e_{1}(h)+\frac{3}{2}%
e_{1}Ve_{1}V(h).
\end{eqnarray*}
So we have 
\begin{eqnarray}
\frac{d^{2}}{dt^{2}}|_{t=0}E_{1}(\Sigma _{t}) &&=\sqrt{2}\int_{\Sigma
_{\lbrack 1/\sqrt{2}]}}[\frac{1}{2}\frac{d}{dt}e_{1}(|H_{cr}|\mathfrak{f})+%
\frac{9}{4}\frac{d}{dt}V(\alpha )+\frac{3}{4}\frac{dH}{dt}]h\Theta \wedge
e^{1}  \notag \\
&&=\frac{\sqrt{2}}{4}\int_{\Sigma _{\lbrack 1/\sqrt{2}%
]}}[e_{1}e_{1}e_{1}e_{1}(h)+3e_{1}Ve_{1}V(h)+7e_{1}e_{1}(h)+9VV(h)+12h]h%
\Theta \wedge e^{1}.  \label{SVF1}
\end{eqnarray}

Suppose either $f_{1}$ or $f_{2}$ has compact support in the nonsingular
domain of $\Sigma .$ Then we have (see \cite[Lemma 7]{CYZ18}) 
\begin{equation}
\int_{\Sigma }f_{1}e_{1}(f_{2})\theta \wedge e^{1}=\int_{\Sigma
}-[e_{1}(f_{1})+2\alpha f_{1}]f_{2}\Theta \wedge e^{1},  \label{e-23}
\end{equation}%
\begin{equation}
\int_{\Sigma }f_{1}V(f_{2})\theta \wedge e^{1}=-\int_{\Sigma
}[V(f_{1})-\alpha Hf_{1}]f_{2}\Theta \wedge e^{1}.  \label{e-24}
\end{equation}
In particular, on the Clifford torus $\Sigma _{\lbrack 1/\sqrt{2}]}$, $%
\alpha =0$ so (\ref{e-23}) and (\ref{e-24}) are reduced to%
\begin{equation*}
\int_{\Sigma _{\lbrack 1/\sqrt{2}]}}f_{1}e_{1}(f_{2})\Theta \wedge
e^{1}=-\int_{\Sigma _{\lbrack 1/\sqrt{2}]}}e_{1}(f_{1})f_{2}\Theta \wedge
e^{1},
\end{equation*}%
\begin{equation*}
\int_{\Sigma _{\lbrack 1/\sqrt{2}]}}f_{1}V(f_{2})\Theta \wedge
e^{1}=-\int_{\Sigma _{\lbrack 1/\sqrt{2}]}}V(f_{1})f_{2}\Theta \wedge e^{1}
\end{equation*}%
\noindent respectively. Note also that on the Clifford torus $h=f-\alpha
g=f, [e_1,V]=0$, and $V=T$. Then for $\Sigma _{0}=\Sigma _{\lbrack 1/\sqrt{2}%
]}:$ the Clifford torus, by applying the above integration by parts
formulas, we reduce (\ref{SVF1}) to%
\begin{equation*}
\frac{d^{2}}{dt^{2}}|_{t=0}E_{1}(\Sigma _{t})=\frac{\sqrt{2}}{4}\int_{\Sigma
_{\lbrack 1/\sqrt{2}%
]}}[e_{1}e_{1}(f)^{2}+3e_{1}T(f)^{2}-7e_{1}(f)^{2}+9fTT(f)+12f^{2}]\Theta
\wedge e^{1}.
\end{equation*}

\noindent We have obtained (\ref{SVF}), completing the proof of Theorem \ref%
{T-2nd-var}. $\hfill \Box$

\end{document}